\documentclass[11pt]{article}

\usepackage{amssymb,amsmath,amsthm}
\usepackage{thmtools}
\usepackage{bbm}
\usepackage[noend]{algpseudocode}
\usepackage{graphicx}
\usepackage{verbatim}
\usepackage{url,xspace,hyperref}
\usepackage{cite}
\usepackage{caption}
\usepackage{multirow}
\usepackage{multicol}
\usepackage{latexsym}
\usepackage{amsmath,amssymb,enumerate}
\usepackage{xspace}
\usepackage{algorithm,algorithmicx}
\usepackage{float}
\usepackage{xcolor}
\usepackage{mathrsfs}
\usepackage{cleveref}
\usepackage{float}
\usepackage{mathtools}
\usepackage{bm}
\usepackage{tikz}
\usepackage{caption}
\usepackage{multirow}
\usepackage{makecell}
\usepackage{enumitem}  
\usepackage{subcaption}
\usepackage{booktabs}
\usepackage{siunitx}
\usepackage{tabularx}
\usepackage{array}
\usetikzlibrary{positioning,decorations.pathreplacing,quotes}
\usetikzlibrary{bayesnet}

\tikzset{tick/.style={draw, minimum width=0pt, minimum height=2pt, inner sep=0pt, label=below:$#1$},
    tick/.default={}}

\usepackage{fullpage}
\usepackage{geometry}
\usepackage{setspace}
\usepackage{changepage}
\usepackage{titlesec}

\newtheorem{theorem}{Theorem}[section]
\newtheorem*{theorem*}{Theorem}

\newtheorem{proposition}[theorem]{Proposition}
\newtheorem{claim}[theorem]{Claim}
\newtheorem{lemma}[theorem]{Lemma}

\newtheorem{example}{Example}[section]

\newtheorem{remark}{Remark}[section]

\newcounter{note}[section]

\makeatletter
\newcommand{\inlineitem}[1][]{%
\ifnum\enit@type=\tw@
    {\descriptionlabel{#1}}
  \hspace{\labelsep}%
\else
  \ifnum\enit@type=\z@
       \refstepcounter{\@listctr}\fi
    \quad\@itemlabel\hspace{\labelsep}%
\fi}
\makeatother

\def\sse{\subseteq}

\newcommand{\pr}{\mathbf{P}} 
\newcommand{\E}{\mathbb{E}} 
\newcommand{\ZZ}{\mathbb{Z}} 
\newcommand{\x}{\mathbf{x}}

\newcommand{\C}{\mathcal{C}}

\newcommand{\ZF}{\mathtt{ZF}}

\newcommand{\OPT}{\mathtt{OPT}}
\newcommand{\LP}{\text{LP}}

\newcommand{\ALG}{\mathtt{ALG}}

\newcommand{\eA}{\mathtt{A}}
\newcommand{\SF}{\mathtt{DF}}

\newcommand{\run}{\mathtt{Select}}

\newcommand{\NC}{\mathtt{NC}}

\newcommand{\ignore}[1]{}

\newcommand{\ssx}{{\fontfamily{cmtt}\selectfont SSE}\xspace}

\newcommand{\algval}{\mathtt{ALG}}

\newcommand{\lp}{\textbf{\textsc{LP$_\text{UB}$}}}

\newcommand{\lpddl}{\textbf{\textsc{LP-ddl}}}
\newcommand{\lpddla}{\textbf{\textsc{ddl(a)}}}
\newcommand{\lpddlb}{\textbf{\textsc{ddl(b)}}}

\renewcommand{\cite}[1]{\citep{#1}}

\usepackage[round]{natbib}
\definecolor{mydarkblue}{rgb}{0,0.08,0.45}
\definecolor{mydarkred}{rgb}{0.75, 0.27, 0.34}
\hypersetup{
    colorlinks=true,
    citecolor=mydarkblue,
    linkcolor=mydarkred,
    }

\title{Sequential Selection with Expirations}

\author{Yihua Xu\footnote{Department of Computational Applied Mathematics and Operations Research / Ken Kennedy Institute, Rice University, USA. Email: yx67@rice.edu.} \and Rohan Ghuge\footnote{Department of Information, Risk, and Operations Management, McCombs School of Business, University of Texas at Austin, USA. Email: rohan.ghuge@mccombs.utexas.edu.}  \and Sebastian Perez-Salazar\footnote{Department of Computational Applied Mathematics and Operations Research / Ken Kennedy Institute, Rice University, USA. Email: sp121@rice.edu.}}

\begin{document}

\maketitle
\thispagestyle{empty}

\begin{abstract}
    {\bf Problem definition:} 
    Many decision-making environments involve opportunities that disappear if not acted upon quickly and whose evaluation requires uncertain amounts of time. 
    Examples include on-demand service platforms, financial portfolio selection, and drug discovery, where high‑value options may expire if delayed and evaluating any option consumes scarce time and resources. 
    We model these settings through the \emph{sequential selection with expirations} (\ssx) problem: a decision-maker must evaluate options one at a time, each option has a random processing time and a random expiration time, and value is earned only if evaluation finishes before expiration. The goal is to design a policy that maximizes expected total value.
    {\bf Methodology/results:} 
    We develop a \emph{time-indexed} linear programming (LP) relaxation that upper bounds the performance of any online policy for \ssx.
    Using this LP, we design a polynomial-time $\frac{1}{2}\cdot\left(1-\frac{1}{e}\right)$-approximation algorithm. 
    We also establish hardness results showing that our analysis is nearly tight.
    For the special case where processing times are independent and identically distributed (i.i.d.), we prove that a simple greedy policy,
    which selects the highest-valued available option, achieves a $\frac{1}{2}$-approximation, and that this guarantee is the best possible for this class of policies.
    Our framework extends naturally to variants with deadlines and knapsack constraints, for which we provide analogous approximation guarantees.
    {\bf Managerial implications:} 
    \ssx captures the natural tension in sequential decision-making, where decision-makers must weigh the benefit of evaluating an option that appears profitable now—without knowing how long the evaluation will take—against the risk of missing other opportunities that may expire during that time. Our results show that being fully greedy in the i.i.d. evaluation-time setting, or greedy and LP-informed in the general case, can effectively alleviate these tensions. These approaches, which are easy to implement, are robust to natural model extensions extending the applicability of \ssx. Two computational studies---one in active learning and one using a call-center datasets---demonstrate the practical performance of our solutions and provide decision-makers with guidance on balancing speed and value under uncertainty and expiration.
\end{abstract}

\newpage

\setcounter{page}{1}
\section{Introduction}\label{sec:intro}

In many practical applications, such as on-demand service platforms offering client-facing services, e-commerce firms testing different promotional strategies, or investors selecting assets for a financial portfolio, the decision-maker (DM) is faced with an abundance of options.
Each option requires some time to search over and not all options may be available throughout the search process.
Much of the traditional literature assumes that all options remain available indefinitely (see, for example, prior work on stochastic knapsack models~\citep{dean2008approximating,ma2018improvements,li2023fully}, virtual machine allocation~\citep{cohen2019overcommitment}, and models of sequential search~\citep{weitzman1978optimal,epstein2024selection}).
While such models are useful for controlled, static environments, they fail to capture urgency, competition, decay, and other real-world dynamics.
For example, in on-demand service platforms, client requests may expire if not matched quickly; in e-commerce, promotional strategies may lose relevance as consumer behavior or product inventories change; and, in financial portfolio selection, investment opportunities may no longer be viable if not acted upon promptly. 

Although selecting an option yields a reward, in many practical settings, it often requires spending a (possibly stochastic) amount of time before the reward can be obtained.
For instance, on-demand service platforms must match client requests to suitable providers but matching a client request generates revenue for the platform;  marketing teams must invest resources to evaluate the effectiveness of new promotional strategies, which may simultaneously improve sales; and financial firms must conduct due diligence before investing, but successful investments can yield a substantial financial return.
Consequently, it may be infeasible to fully evaluate—and benefit from—every option before some become unavailable.
This can result in a loss of revenue in several settings such as online service platforms, e-commerce firms, or financial portfolio management.
It is thus important for the DM to design prioritization policies for selecting from the available set of options to maximize the total reward obtained.

Motivated by this, we study a \emph{sequential selection} problem in which the DM is presented with $n$ options. The DM can select the options one at a time, and each option $j \in [n]$ has a value $v_j$ associated with it and an initially unknown ``evaluation'' or ``processing'' time, distributed according to a known probability distribution. 
While the DM evaluates a particular option, the remaining options may become unavailable or ``expire'' with some (known) probability.
This process repeats until there are no options remaining. 
The objective of the DM is to maximize the total value obtained from the selected options. 
We defer the formal description of the problem, which we term as \emph{sequential selection with expirations} (\ssx), to \Cref{sec:model}. 

Our problem is related to two streams of research: \emph{stochastic combinatorial optimization} and \emph{selection with uncertainty}.
Solutions to problems in stochastic combinatorial optimization are sequential decision processes, where each selection is made sequentially, and possibly adaptively, using all the information from previous selections. 
The optimal solution can be represented by a Markov decision process (MDP), and often suffers from the curse of dimensionality. 
However, many recent works have obtained polynomial-time approximation algorithms for these problems (see, for e.g., \citet{goemans2006stochastic, dean2008approximating, guha2007approximation, bansal2012lp, asadpour2016maximizing,  JiangLL+20, ghuge2024nonadaptive}).
The notion of selection with uncertainty has been used to generalize many fundamental search problems, especially related to sequential search; see, for e.g., ~\citet{kesselheim_mdp, purohit2019hiring, epstein2024selection, perez2024robust, brown2024sequential}.
More recently, the idea of uncertain selection has also been studied in the context of \emph{multistage optimization} by \citet{cygan2013catch} and \citet{segev2024near}. 

The main result of our paper is a polynomial-time constant-factor approximation algorithm for \ssx relative to the optimal online algorithm. 
See \Cref{appendix:benchmark} for a detailed justification of this choice of benchmark.
Henceforth, unless stated otherwise, all performance guarantees are measured relative to the optimal online policy.
Below, we provide a summary of our contributions.

We initiate the study of \emph{sequential selection with expirations}, a model that captures sequential selection settings where (1) the processing time of each option is an \emph{independent} random variable, and (2) each option \emph{independently} stays in the system for a random amount of time, both following known distributions.
We develop a compact linear program that upper bounds the performance of any online policy, and leverage its optimal solution to design a polynomial-time $\frac{1}{2}\cdot\left(1-\frac{1}{e}\right) \approx 0.316$-approximation algorithm. 
We demonstrate the generality of this approach by analyzing natural extensions of \ssx, such as incorporating a knapsack constraint or imposing hard deadlines on the options. For each of these extensions, we design algorithms that achieve approximation guarantees comparable to those for the vanilla \ssx.
Furthermore, we analyze \ssx under the special case where processing times are independent and identically distributed. We prove that, in this case, the greedy policy which selects the highest-valued available option whenever possible achieves a $\frac{1}{2}$-approximation to the online optimum. We further show that this bound is tight for this class of greedy policies.
Finally, we extensively evaluate our algorithms on instances generated from synthetic and real-world datasets to validate our theoretical results.

\subsection{Problem Definition}\label{sec:model}
In the \emph{sequential selection with expirations} (\ssx) problem,  we are given $n$ options. 
Each option $j \in [n] = \{1, 2, \ldots, n\}$ has a deterministic value $v_j > 0$, and is associated with two random variables $S_j$ and $E_j$.
We use $S_j$ and $E_j$ to denote the uncertain evaluation time and the uncertain expiration time of option $j$, respectively.
We assume that all random variables are independent with known probability distribution; however, the realization remains unknown. 

The random process we consider evolves along a sequence of discrete stages or time units, beginning with $t=1$. To formalize the system dynamics, we use $R_t$ to denote the set of options available at time $t \in \mathbb{N}$. 
We say that the DM is \emph{busy} when an option is being evaluated, else the DM is \emph{free}. 
At the beginning, all options are available; i.e., $R_1 = [n]$, and the DM is free.
Suppose the DM is free at time $t$. Then, the following sequence of steps occurs. 
\begin{enumerate}
    \item The DM selects an option $j \in R_t$, which causes the DM to be busy for $S_j$ units of time. 
    Once evaluation is complete, the corresponding value $v_j$ is obtained.
    
    \item In the meanwhile, other options $j' \in R_t$ may expire according to the r.v.s $E_{j'}$. 
\end{enumerate}

\noindent We assume that a selected option can neither be preempted (that is, it cannot be paused and re-evaluated later) nor expire.  We use $\overline{F}_{S_j}(t) := \Pr(S_j > t)$ and $p_j(t) := \Pr(E_j > t)$ to denote the tail distributions of $S_j$ and $E_j$, respectively. 
For simplicity, we will assume that $\overline{F}_{S_j}(T) = 0$ for some $T \in \mathbb{N}$, and that $p_j(t) > 0$ for all $t \geq 1$ for each $j \in [n]$. 
We note that both of these assumptions can be made at the expense of an arbitrarily small loss in the objective value (see \Cref{app:bounded_support} and~\ref{app:unbonuded-support} for details). 

A solution to \ssx is a sequential selection policy that, given a set of available options and the current time $t \leq T$, decides which option to select next (when the DM is free). 
The objective in \ssx is to design a sequential selection policy in order to maximize the expected value obtained from selected options.
We use $\OPT$ to denote the optimal policy. It is easy to check that $\OPT$ can be computed recursively as follows, 
$$V(t, R) = \max_{j \in R} \left\{ \sum_{\tau =1} ^{T} \Pr\left(S_j = \tau\right) \cdot \left(v_j + \sum_{R' \subseteq R \setminus \{j\}} \Pr\left(R',t+\tau\right) \cdot V\left(t + \tau, R'\right) \right) \right\}, $$
where $V(t, R)$ denotes the maximum expected value that can be obtained starting from time $t$ when options $R \sse [n]$ are still available.
Above, $\Pr(R',t+\tau)$ denotes the probability that options $R' \sse R\setminus \{j\}$ are available at the start of time $t+\tau$.
We note that the above dynamic program suffers from the curse of dimensionality, making it computationally infeasible for large instances. 
In fact, \ssx captures the well-studied NP-hard knapsack problem; thus, it is unlikely that a succinct characterization of the optimal policy exists (see \Cref{appendix:benchmark} for details).

\medskip
\noindent \textbf{Extensions.} We consider two extensions of \ssx that are motivated by real-world applications. In the first extension, we consider \emph{options with deadlines}: each option $j \in [n]$ is associated with a deadline $D_j$, and a value is obtained only if the option is evaluated completely before its deadline. 
Formally, if an option selected at time $t$ requires $\tau$ time for evaluation such that $t + \tau \leq D_j$, then the policy obtains value $v_j$; otherwise, it receives zero value.
In the second extension, we introduce a \emph{global knapsack constraint} on the selection process. Each option $j \in [n]$ is assigned a weight $w_j \geq 0$, and the total weight of selected options must not exceed a given capacity $W$. 

\begin{remark}[Uncertain and Time Dependent Values]
We note that our framework can naturally accommodate \emph{uncertain values}, as long as their distributions are known in advance. 
In such cases, the value $v_j$ associated with each option $j$ can be replaced by its expected value without affecting the analysis.
Moreover, our techniques extend to settings with certain forms of \emph{time-dependent} values, specifically when the value associated with an option depends on the time it is selected (rather than when it completes evaluation). 
When values are obtained immediately upon selection, incorporating time-dependent value functions $v_j(t)$ that vary with the selection time $t$ is straightforward and does not alter the structure of our algorithms or analysis. 
\end{remark}

\subsection{Results and Techniques} \label{sec:results-techniques}

Our first main result establishes a constant-factor approximation for the sequential selection with expiration problem.

\begin{theorem}\label{thm:main-approx}
    For any $\epsilon > 0$, there exists a polynomial-time algorithm that achieves a $\textstyle \left(0.316-\epsilon\right)$-approximation  for the sequential selection with expiration problem, relative to the online optimum.
\end{theorem}

This result builds on and significantly extends the prior work of~\citet{cygan2013catch} and~\citet{segev2024near}.
In particular, while those works studied selection problems under geometrically distributed departures and instantaneous evaluations, our model generalizes this setting in two key ways: (1) by incorporating stochastic service times, and (2) by allowing arbitrary expiration distributions.
Concretely, our result resolves an open question posed by~\citet{segev2024near}, who asked whether similar approximation guarantees could be obtained in the more general setting with arbitrary departure times.

In order to obtain our result, we first introduce a time-indexed linear program (LP) that upper bounds the expected total value of any policy, including that of the optimal policy (see~\Cref{thm:main-upperbound-opt}). 
This LP captures feasibility \emph{in expectation}, using marginal decision variables $x_{j, t}$ that represent the probability of selecting option $j$ at time $t$.
However, directly rounding an optimal LP solution faces two key challenges: (1) selecting an option at time $t$ affects the probability that the DM is available at a future time $\tau > t$, creating dependencies across time steps, and 
(2) na\"{i}ve independent rounding may result in multiple options being selected simultaneously, whereas the DM can select only one option at a time.
We get around these challenges as follows. To address the first issue, we scale down the LP solution to ensure that, at any time $t$, the DM remains free with a sufficiently large probability. In order to balance this with obtaining a constant fraction of the LP value at each time step, we select a scaling factor of 1/2.
To get around the other issue, 
we introduce the notion of a \emph{consideration set}. Whenever the DM is free, we construct this set by sampling from the scaled LP solution, restricted to the currently available options.
From this set, we then greedily select the highest-valued option. This final greedy step ensures that, in expectation, we obtain at least a $(1-1/e)$ fraction of the value obtained by the LP at that corresponding time step. 

Lastly, we  demonstrate the limitation of our algorithm and the analysis. We show that our approximation guarantee is nearly tight: we construct an instance in which our algorithm achieves at most a $(1 - 1/\sqrt{e}) \approx 0.393$ fraction of the LP optimum demonstrating that our analysis cannot be significantly improved in general.
We also provide a bound of $(1-1/e)$ on the integrality gap of our LP, establishing a fundamental limit on what any LP-based algorithm that uses our LP can achieve.
See \Cref{lem:approx_hard} and \Cref{lem:lp-gap} for details.


We obtain a better approximation guarantee when the evaluation times of options are independently and identically distributed.

\begin{theorem}\label{thm:greedy}
There exists a polynomial-time algorithm that achieves a $0.5$-approximation for the sequential selection with expiration problem, relative to the online optimum, when the evaluation times are independently and identically distributed.
\end{theorem}

We obtain this result using a natural greedy policy: select the highest-valued available option whenever the DM is free.
While the policy is simple to state, analyzing its performance requires a careful combination of coupling and charging arguments.
At a high-level, we use a coupling argument to ensure that the DM is free at the same time steps under both the greedy and optimal policies.
Then, through a charging argument, we show that, whenever the greedy policy selects an option, twice the value collected by the greedy policy suffices to account for the 
value collected by the optimal policy at the same time and one step in the future. 
We present more details in \Cref{greedy_section}. 
Finally, we note that our analysis for the greedy algorithm is tight, as demonstrated by the following example.

\begin{example}\label{ex:greedy_bad_iid}
Consider the following instance with $n = 2$ options. 
We set the parameters of the options as follows: $v_1 = 1+\varepsilon$, $E_1 = 3$ with probability $1$; $v_2 = 1$ and $E_2 = 2$ with probability $1$.
Let $S_1$ and $S_2$ be $1$ with probability $1$.
Observe that the greedy algorithm (serving the highest valued option) obtains 
$(1+\varepsilon)/(2+\varepsilon)$ times the value obtained by an optimal policy.
\end{example}

Due to the inherent simplicity and practicality of greedy solutions, it might be tempting to formulate greedy algorithms for \ssx.
However, as the subsequent examples illustrate, these intuitive, 
natural greedy approaches, such as prioritizing options based on their value or on the ratio of value to expected evaluation time, do not yield favorable results.

\begin{example}\label{ex:greedy_bad_noniid}
Consider the greedy algorithm that selects the option with the largest value.
Consider the following instance of $n$ options. We set the parameters of option $1$ as follows: $v_1 = 1+\varepsilon$, $E_1 = n+2$
and $S_1 = n$ with probability $1$. For each option $j \in \{2, \ldots, n\}$, we have 
$v_j = 1$, $E_j = n+1$ and $S_j = 1$ with probability $1$. 
Observe that the greedy algorithm 
selects option $1$ at time $t=1$ and
obtains a value of $1+\varepsilon$, 
while the optimal solution obtains a total value of $n+\varepsilon$ by first serving options $2, \ldots n$, and then option $1$. 
\end{example}

\begin{example}
Consider the greedy algorithm that selects the option with the largest ratio of value to expected evaluation time. Consider the behavior of this algorithm on the following instance with $n=2$. We have $v_1=1+\varepsilon$, $E_1 = K+3$ and $S_1=1$ with probability $1$, and $v_2=K$, $E_2 = 2$ and $S_2 = K$ with probability $1$. 
The greedy algorithm selects option $1$ at time $t=1$, and loses out on the value of option $2$ (which leaves the system at $t = 2$). Thus, the greedy algorithm obtains a value of $1 + \varepsilon$,
while the optimal solution obtains a total value of $K+1+\varepsilon$.
\end{example}

Our techniques naturally extend to handle richer variants of the sequential selection with expiration problem introduced in \Cref{sec:model}. 
We first consider the setting with option-specific deadlines, where each option must be fully evaluated before a prescribed deadline in order to obtain any value.
This model generalizes the classic stochastic knapsack problem~\citep{dean2008approximating} by allowing options to leave the system.
Our main result in this setting is as follows.
\begin{theorem}\label{thm:main-deadlines}
    For any $\epsilon > 0$, there exists a polynomial-time algorithm that achieves a $\textstyle \left(0.316-\epsilon\right)$-approximation  for the \ssx with deadlines problem, relative to the online optimum.
    Moreover, if the evaluation times are independently and identically distributed, there exists a polynomial-time algorithm that achieves a $0.5$-approximation.
\end{theorem}

Our next extension introduces a global knapsack constraint on the options that the decision-maker (DM) can select. In this setting, each option $j$ has a weight $w_j\geq 0$, 
and the total weight of selected options cannot exceed a predetermined capacity $W$. 
This also captures cardinality constraints as a special case, where $w_j=1$ for all $j\in [n]$ and $W=k\in \ZZ$. 
Letting $w_{\max} = \max_j w_j$, our main result in this setting is as follows.

\begin{theorem}\label{thm:main-knapsack}
    For any $\epsilon > 0$, there exists a polynomial-time algorithm that achieves a \\ \(\left((1 - e^{-\Omega(W^2/nw_{\max}^2)}) \cdot 0.316-\epsilon\right)\)-approximation  for \ssx with a knapsack constraint, relative to the online optimum.
\end{theorem}

For the special case of cardinality constraints, we obtain the following result.
\begin{theorem}
    For any $\epsilon > 0$, there exists a polynomial-time algorithm that achieves a \\ $\left((1 - e^{-\Omega(k)}) \cdot 0.316-\epsilon\right)$-approximation  for \ssx with a cardinality constraint, relative to the online optimum.
    Moreover, if the evaluation times are independently and identically distributed, there exists a polynomial-time algorithm that achieves a $0.5$-approximation.
\end{theorem}

{
\subsection{Selected Applications}

We demonstrate the practical applicability of the sequential selection with expiration problem using an application in \emph{active search}, 
where one is given a pool of unlabeled points with unknown binary labels, and the goal is to identify as many positive points as possible under a budget constraint, or to find a target number of positives using as few queries as possible~\citep{garnett2012bayesian, jiang2017efficient, jiang2019cost}. 

Active search arises in domains where input data is easy to obtain, but label acquisition is costly or time-consuming. 
For instance, in drug discovery, it is relatively easy to generate molecular structures, but testing their efficacy (i.e., whether they are ``positive'' hits) requires expensive experiments.
An early example of this arose in the \textsc{Meta-DENDRAL} project where the goal was to learn rules that could predict the behavior
of molecules inside a mass spectrometer~\citep{buchanan1978model, lindsay1980applications}. 
A similar bottleneck appears in product design, where it is easy to prototype ideas but costly to test their actual impact through A/B testing or consumer trials.
The canonical active search setup involves a finite set $\mathcal{X}$ of items, among which an unknown subset $\mathcal{R} \subseteq \mathcal{X}$ are ``positives.'' 
The aim is to either recover as many elements of $\mathcal{R}$ as possible given a fixed query budget, or to reach a fixed number of positives using as few queries as possible. Each query reveals the label of a selected item but consumes time or resources.

We argue that our sequential search with expiration framework offers a more realistic abstraction for such applications. 
In many practical settings:
\begin{enumerate}
    \item \emph{Evaluation incurs time-dependent opportunity costs.} For instance, drug assays or user experiments require non-negligible durations, during which other opportunities may be lost.

    \item \emph{Items may become unavailable.} For example, due to expiration of relevance, depletion of experimental resources, or shifting external constraints.

    \item \emph{Prioritization under uncertainty is crucial.} High-value items (e.g., promising drug candidates or novel features) may require longer evaluations, risking the loss of other options.

\end{enumerate}

Our model captures these dynamics by incorporating both uncertain evaluation times and stochastic expirations, offering a unified framework that 
captures many practical considerations.
This yields a natural generalization of active search, and suggests new algorithmic strategies with provable guarantees under more realistic assumptions.

As a proof of concept for our framework, one of our computational experiments constructs an instance using data from a publicly available large language model (LLM) leaderboard, which reports performance on the GPQA benchmark, a challenging benchmark designed to test graduate-level proficiency in multiple subjects.
Here, an instance of a model corresponds to an ``option''.
The value of a model is given by its reported accuracy or performance score on GPQA, while the cost and evaluation time of each model are derived from metadata such as model size, inference latency, and the number of tokens processed per second (a proxy for how quickly the model can be evaluated).
Our model enables us to evaluate policies that balance high performance potential with fast evaluation, and compare them against natural baselines that greedily select the highest-scoring or fastest models. 
See \Cref{sec:num_exp} for further details.

}

\paragraph{Organization.} 
The remainder of the paper is organized as follows.
We discuss related work in \Cref{sec:related-work}.
We present our LP-based algorithm and prove \Cref{thm:main-approx} in \Cref{sec:main-alg}.
In \Cref{greedy_section}, we show that a natural greedy policy achieves a $0.5$-approximation for instances with i.i.d.\ evaluation times (thus proving \Cref{thm:greedy}). 
We present our computational results in \Cref{sec:num_exp} and conclude in \Cref{sec:conclusion}.


\section{Related Work}\label{sec:related-work}
The sequential search problem has been extensively studied in operations research, economics, and computer science (see, e.g., \citet{baucells2024search, brown2024sequential, kleinberg2016descending}). Classical models typically involve a decision-maker who sequentially evaluates options under value uncertainty, aiming to maximize the expected (discounted) value from the selected option, while accounting for search costs and operational constraints. Our model bears a loose resemblance to Weitzman’s classical \emph{Pandora’s box problem}~\citep{weitzman1978optimal}, where exploring an option incurs an explicit \emph{opening cost}. 
In our setting, the evaluation time functions as an implicit cost of exploration, since allocating time to evaluate one option inherently limits the opportunity to examine others within the same time horizon. 
Numerous extensions have been proposed to enrich this framework. 
For instance, one line of work explores settings in which options may be selected without direct exploration~\citep{beyhaghi2023pandora, alaei2021revenue}, trading off savings in search cost for the risk of accepting lower-value options. 
In our setting, the trade-off arises from temporal constraints: while high-value options are desirable, their evaluation may consume more time, during which other options expire. 
This introduces a form of opportunity cost driven by evaluation delays. 
Such role of time in search models has been a central theme in the literature, with important contributions addressing simultaneous~\citep{chade2006simultaneous}, parallel~\citep{vishwanath1992parallel}, and ordered~\citep{armstrong2017ordered} search processes.

Our problem bears structural resemblance to scheduling problems. From the job scheduling perspective, a single server processes $n$ queued jobs. When idle, it selects a job which yields value, and becomes busy for a random service time. During this period, the remaining jobs may depart. 
Our problem shares similarities with the \emph{interval scheduling} problem (see surveys~\cite{kolen2007interval} and~\cite{mohan2019review}).
In the interval scheduling problem, we are given a set of jobs with specified starting times, and the goal is to design a policy such that
no two jobs overlaps. In this setup, the performance measures can focus on individual jobs;
for example, we want to run as many jobs as possible, or maximize the total weight of the accepted jobs. 
Several models consider stochastic service times; for instance, see \citet{gittins1979bandit} and \citet{shoval2018probabilistic}. 
Closer to our work is the work of \cite{chen2023stochastic} which considers stochastic service times and abandonment and reach necessary and sufficient conditions for the optimality of a strict priority policy; however, their analysis is restricted to two specific classes of probability distributions.

Exploring options with stochastic expiration (a.k.a.\ departure) has also been studied extensively in the management literature~\citep{puha2019scheduling,kim2018dynamic,zhong2022learning}. 
There is a body of work at the intersection of queuing theory and job scheduling which investigates settings like service-level differentiation~\citep{gurvich2010service}, shortest-remaining-processing-time scheduling policy~\citep{dong2021srpt}, and setup costs~\citep{hyon2012scheduling}, with the aim of designing low-regret policies.

Central to our analysis for general evaluation times is the use of LP relaxations, which have been used extensively in job scheduling \citep{im2015stochastic,mohring1999approximation}, 
stochastic knapsack \citep{blado2019relaxation,dean2008approximating, gupta2011approximation,ma2018improvements} 
and other related problems \citep{aouad2020dynamic,li2023fully}.
Closer to our work is~\cite{cygan2013catch}, where they study instances of \ssx where options expire with a geometric random time, and each option has a deterministic and uniform evaluation time. 
\cite{cygan2013catch} gives an LP-based algorithm that guarantees a $0.709$ approximation; however, their approach can only be applied when the evaluations times are uniform.
A recent work~\citep{segev2024near} provides a quasi-polynomial time approximation scheme for this problem.

Our model exhibits parallels with online bipartite-matching with reusable resources~\citep{dickerson2021allocation, goyal2022asymptotically}.
While their frameworks focus on online arrivals (an aspect absent from our model), our setting incorporates stochastic expirations, which are not captured in theirs.

The sequential search aspect of our framework is closely related to the literature on Markovian multi-armed bandits (MABs) (see, e.g.,~\citet{gittins1979bandit} and \citet{bubeck2012regret}). 
In classical MABs, each arm evolves as a Markov chain and generates rewards based on the current state when pulled, with the goal of allocating a limited number of pulls to maximize cumulative value.
Of particular relevance is the MAB superprocesses model with multi-period actions introduced by~\citet{ma2018improvements}, which generalizes traditional bandits by allowing actions that span multiple time steps. This notion of non-unit processing time closely mirrors our setting, where evaluating an option requires a random amount of time.

However, our model introduces an additional layer of complexity: options may expire during the evaluation of a selected option.
This dynamic is conceptually related to the restless bandits framework, in which arms evolve even when not played~\citep{whittle1988restless}.
The restless bandit framework has been studied extensively, including approximation algorithms in stationary settings~\citep{guha2010approximation, grunewalder2019approximations} and models incorporating irrevocability, limited availability, and costly exploration~\citep{chakrabarti2008mortal, farias2011irrevocable, traca2020reducing, shao2024multi} — features closely aligned with our setting. 
Crucially, our work unifies expiration and non-unit processing within a single approximation framework, extending beyond prior models that typically handle only one of these two complexities.


\section{LP-based Algorithms for \ssx}\label{sec:main-alg}
In this section we present and analyze the LP-based approximation algorithm for \ssx, and prove \Cref{thm:main-approx}.
To recap, a solution to \ssx is a sequential selection policy that, given a set of available options and the current time $t \leq T$, decides which option to select next (when the DM is free). 
The objective is to design a sequential selection policy to maximize the expected value obtained from selected options. In \Cref{subsec:lp_relaxation}, we present a linear program that provides an upper bound on the value obtained by any online policy.  
We use this LP to design an approximation algorithm for \ssx in \Cref{sec:approx-alg}. 
In \Cref{sec:limitations}, we present two results that
highlight the limitations of our algorithm and the LP-based analysis. 
The approximation algorithms for the extensions of \ssx are almost identical: the changes are explained in \Cref{sec:extension}.

\subsection{The Linear Programming Relaxation}\label{subsec:lp_relaxation}
We now present a \emph{time-indexed} linear program that upper bounds the value of any optimal policy for \ssx. 
Recall that we use $\overline{F}_{S_j}(t) := \Pr(S_j > t)$ and $p_j(t) := \Pr(E_j > t)$ to denote the tail distributions of $S_j$ and $E_j$, respectively. Furthermore, for each $j \in [n]$, we have $\overline{F}_{S_j}(T) = 0$ for some $T \in \mathbb{N}$, and  $p_j(t) > 0$ for all $t \geq 1$.
We define variables $x_{j, t}$ to denote the \emph{probability that a policy selects option $j$ at time $t$}. The LP is as follows:
\begin{align}
    \text{maximize} \quad & \textstyle \sum_{t =1}^{nT}\sum_{j=1}^n v_j \cdot x_{j,t}   \notag \\ 
     \text{s.t.} \ \ \forall j \in [n] \quad & \textstyle \sum_{t= 1}^{nT} \frac{x_{j, t}}{p_j(t)}\leq 1  \tag{\lp}\label{lp:LP1} \\
     \forall t \in [nT] \quad & \textstyle \sum_{\tau \leq t}\sum_{j=1}^n  \overline{F}_{S_j}(t-\tau) \cdot  x_{j, \tau} \leq 1   \notag \\
      \forall j \in [n], \forall t \in [nT] \quad &   x_{j, t} \geq 0.  \notag
\end{align}

The two sets of constraints in \ref{lp:LP1} are deduced from natural constraints that any policy for \ssx must satisfy: (1) no option can be selected more than once, and (2) at any given time, the DM can evaluate at most one option. 
We note that $t \in \{1, 2, \ldots, nT\}$ since $\overline{F}_{S_j}(T) = 0$ for $T \in \mathbb{N}$ for all $j \in [n]$.
The LP has a total of $O(n^2T)$ variables and constraints, and is therefore solvable in polynomial time provided that $T$ is polynomial in $n$.
We begin by showing that \ref{lp:LP1} provides an upper bound on the optimum. 

\begin{theorem}\label{thm:main-upperbound-opt}
    Given an instance of sequential selection with expiration, \ref{lp:LP1} upper bounds the expected value of the optimal online policy. 
\end{theorem}

\begin{proof}
Let $\Pi$ denote the optimal online policy for the given instance of \ssx. Let $x^*_{j, t}$ denote the probability that $\Pi$ selects option $j$ at time $t$. Then, the objective value of $\x^* :=\{x^*_{j,t}\}_{j \in [n], t \in [T]}$ is $\sum_{t= 1}^{nT} \sum_{j=1}^n v_j \cdot x^*_{j, t}$, which is equal to the expected value obtained by the optimal policy $\Pi$. 
To finish the proof, 
we need to show that $\x^*$ is feasible for \ref{lp:LP1}. 
Towards this end, let $Y_{j, t}$ denote the event that $\Pi$ selects option $j$ at time $t$.
Fix $j \in [n]$ and $t\geq 1$. Since an option can only be selected if it is still available, we have
\begin{align*}
  \textstyle  \frac{x^*_{j,t}}{p_j(t)} \,\, & = \,\, \textstyle  \Pr( Y_{j, t} \mid E_j > t)
\,\, = \,\, 1 - \Pr( \overline{Y}_{j, t} \mid E_j > t) \,\, \leq \,\, 
      1 - \sum_{\tau < t} \Pr(Y_{j, \tau} \mid E_j > t) \\
    & = \,\, \textstyle 1-\sum_{\tau < t} \Pr(Y_{j, \tau} \mid E_j > \tau) 
    \,\, = \,\, 1- \sum_{\tau < t} \frac{x^*_{j,\tau}}{p_j(\tau)},
\end{align*}
where the first inequality holds since $\Pi$ can select option $j$ only once, and
the penultimate equality uses the fact that the policy's decision at time $\tau$ depends only on information up to time $\tau$, so conditioning on $E_j > t$ is equivalent to conditioning on $E_j > \tau$ for $\tau < t$. On-rearranging the above inequality (and applying it to $t = nT$), we conclude that $\x^*$ satisfies the first set of constraints of \ref{lp:LP1}.
To show that $\x^*$ satisfies the second set of constraints, we begin by defining $\mathbb{I}^{\ \text{start}}_{j, t}$ as an indicator random variable that is $1$ if $\Pi$ selects option $j$ at time $t$, 
and $0$ otherwise. 
At any given time $t \geq 1$, the DM can either (i) select an available option, or (ii) continue evaluating an option selected earlier. 
In the second case, option $j$ ``blocks'' time $t$ if it was selected at time $\tau$ {\it and} its evaluation time was greater than $(t-\tau)$: let $S_{j, t-\tau}$ be an indicator denoting this quantity. 
Thus, for any $t \geq 1$, we have $\sum_{j \in [n]}\mathbb{I}^{\ \text{start}}_{j, t} + \sum_{\tau < t}\sum_{j \in [n]}  \mathbb{I}^{\ \text{start}}_{j, \tau} \cdot S_{j, t-\tau} \leq 1$. On taking expectations, we get
\[
\sum_{j \in [n]}\E[\mathbb{I}^{\ \text{start}}_{j, t}] + \sum_{\tau < t}\sum_{j \in [n]}  \E[\mathbb{I}^{\ \text{start}}_{j, \tau} \cdot S_{j, t-\tau}] \,\, = \,\, \sum_{j \in [n]}x^*_{j, t} + \sum_{\tau < t}\sum_{j \in [n]}  x^*_{j, \tau} \cdot \overline{F}_{S_j}(t-\tau) \,\, \leq \,\, 1
\]
where the equality uses the fact that the evaluation times are independent of the decisions made by $\Pi$.
Thus, $\x^*$ is a feasible solution for \ref{lp:LP1}, which completes the proof.
\end{proof}

\subsection{The Algorithm and its Analysis} \label{sec:approx-alg}
In this section, we describe and analyze our approximation algorithm for \ssx. 
Let $ \x^*$ denote the optimal solution of \ref{lp:LP1} for a given instance of \ssx. 
At a high level, our algorithm uses the LP solution $\mathbf{x}^*$ to guide option selection whenever the DM is free.
Specifically, if the DM is free at time $t$, the algorithm constructs a \emph{consideration set} $\mathcal{C}^*=\mathcal{C}^*(t)$ consisting of unexpired options that have not been previously considered. 
The algorithm then selects the highest-valued option from $\mathcal{C}^*$, if the set is non-empty.
Before describing how options are added to $\C^*$, we define the following events for all $j \in [n]$ and $t \geq 1$.
\begin{itemize}
    \item $\eA_{j, t}$: option $j$ is available (i.e., unexpired) at time $t$

    \item $\SF_t$: DM is free at time $t$

    \item $\NC_{j,t}$: option $j$ has not been considered before time $t$
\end{itemize}
Using this, we define $f_{j, t} = \Pr(\NC_{j, t}, \, \SF_t \mid \eA_{j, t})$. In words, $f_{j, t}$ represents the probability that option $j$ has not been considered before time $t$ and that the DM is free at time $t$, conditioned on $j$ being available at that time. 
Then, when the DM is free at time $t$, 
the algorithm adds option $j$ to $\C^*$ independently with probability $\textstyle \frac{x_{j,t}^*}{2\cdot p_j(t)\cdot f_{j,t}}$, provided that $j$ is available and has not been previously considered.
We note that the $f_{j,t}$ values can be computed based on the decision paths of the algorithm up to time $t$. For simplicity, we assume that these values are provided as input to the algorithm (see \Cref{remark:polytime} for further discussion).
See Algorithm~\ref{alg:random-serve} for a formal description of this procedure. 
The main result of this section is as follows.
\begin{theorem}\label{thm:appx-bound}
    \Cref{alg:random-serve} achieves a $\textstyle \frac{1}{2}\cdot \left(1-\frac{1}{e}\right)$-approximation for the sequential selection with expiration problem.
\end{theorem}

\begin{remark}\label{remark:polytime}
    Assuming that the $f_{j, t}$ values can be computed in polynomial time, it is clear that our algorithm runs in polynomial time. 
    However, computing $f_{j,t}$ may involve a sum over exponentially many decision paths, which precludes a polynomial-time implementation.
    To address this, we show in \Cref{app:sampling} that the $f_{j,t}$ values can be efficiently approximated via sampling. Specifically, we prove that $\text{poly}(n, 1/\epsilon)$ samples suffice to estimate each $f_{j,t}$ up to an additive error of $\epsilon$, and our analysis remains valid even when these estimates are used in place of exact values.
    So, the final runtime of our algorithm is poly$\textstyle (n, T, \frac{1}{\epsilon})$, at the cost of an arbitrarily small loss in the objective value.
    Combining this result with \Cref{thm:appx-bound} establishes \Cref{thm:main-approx}. In the following analysis we will assume that the $f_{j, t}$ values are computed exactly. 
\end{remark}

\begin{algorithm}
\caption{An LP-based algorithm for \ssx}
\label{alg:random-serve}
\begin{algorithmic}[1]
\State \textbf{Input:} \ssx instance: value $v_j$, expiration time distribution $p_j(\cdot)$ and evaluation time distribution $\overline{F}_{j}(\cdot)$ for all $j \in [n]$; probabilities $\{f_{j,t}\}_{j \in [n], t \geq 1}$ 
\State $\mathbf{x}^* \gets $ optimal solution to \ref{lp:LP1}
\State $\C \gets [n]$
\For{$t=1,\ldots$}
    \If{DM is free}
        \State $\C^* \gets \emptyset$.
        \For{$j \in \C$}
        \State add $j$ to $\C^*$ w.p. $\textstyle \frac{x^*_{j, t}}{2 \cdot p_j(t) \cdot f_{j,t}}$
        \EndFor
        \State $j_t \gets \arg \max_{j \in \C^*} \{v_j\}$
        \State \textbf{select} $j_t$
    \EndIf
    \State \textbf{update} $\C:$ remove options $\in \C^*$ and expired options
\EndFor
\end{algorithmic}
\end{algorithm}

\paragraph{Proof of \Cref{thm:appx-bound}.}
We begin the proof by showing that \Cref{alg:random-serve} is feasible. 
To prove feasibility, we need to show that  $\textstyle \frac{x^*_{j, t}}{2 \cdot p_j(t) \cdot f_{j,t}} \leq 1$ for all $j \in [n]$ and $t \geq 1$.
The following result provides the key ingredient needed to establish this bound.
\ignore{We prove this statement by induction on $t$. 
By definition $f_{j,1}=1$ and $p_j(1) = 1$ for all $j \in [n]$; hence $\textstyle \frac{x_{j,1}^*}{2 \cdot p_{j}(1) \cdot f_{j,1}} \leq 1$. 
Let $t \geq 2$, and suppose that $\textstyle \frac{x_{j,\tau}^*}{2 \cdot p_j(\tau) \cdot f_{j,\tau}} \leq 1$ for all $\tau \leq t-1$. To complete the proof, we need to show that 
$\textstyle \frac{x^*_{j, t}}{2 \cdot p_j(t) \cdot f_{j,t}} \leq 1$ for all $j \in [n]$.
In order to complete the induction, we need the following result.}

\begin{claim}\label{lem:well-defined}
    Let $\run_{j,t}$ denote the event that  \Cref{alg:random-serve} selects option $j$ at time $t$ and let $\mathtt{C}_{j,t}$ denote the event that option $j$ is in $\C^*$ at time $t$. Then, for any $\tau < t$, we have:
    \begin{enumerate}
        \item $\Pr(\run_{i,\tau}\mid \eA_{j,t}) \leq x_{i,\tau}^*/2$ for all options $i \neq j$, and
        
        \item $\Pr(\mathtt{C}_{j,\tau}\mid \eA_{j,t})\leq \frac{x_{j,\tau}^*}{2\cdot p_j(\tau)}$ for all options $j \in [n]$.
    \end{enumerate}
\end{claim}
\begin{proof}
To select an option $j$ at time $\tau$, it must have been in the consideration set at time $\tau$. Furthermore, an option will not be considered at time $\tau$ if (i) it isn't still available at time $\tau$, (ii) it has already been considered before time $\tau$, or (iii) the DM is not free at time $\tau$. Thus, we get:
\begin{align}
\Pr(\run_{j, \tau}) \,\, &\leq \,\, \Pr(\mathtt{C}_{j, \tau}) \,\,  
= \,\, \Pr(\mathtt{C}_{j, \tau} \mid \eA_{j, \tau}, \SF_\tau, \NC_{j, \tau}) \cdot \Pr(\SF_\tau, \NC_{j, \tau} \mid \eA_{j, \tau}) \cdot \Pr(\eA_{j, \tau}) \notag \\
&= \,\, \frac{x^*_{j, \tau}}{2 \cdot p_j(\tau) \cdot f_{j, \tau}} \cdot f_{j, \tau} \cdot p_j(\tau) \,\, = \,\, \frac{x^*_{j, \tau}}{2} \label{eq:prob-run}
\end{align}
where the penultimate inequality follows from the algorithm design and the definition of $f_{j, \tau}$. 
Using the law of total probability, we can write 
$$\textstyle \Pr(\run_{i,\tau}) =  \textstyle \Pr(\run_{i,\tau} \mid \eA_{j, \tau}) \cdot \Pr(\eA_{j, \tau}) +  \Pr(\run_{i,\tau} \mid \overline{\eA}_{j, \tau}) \cdot \Pr(\overline{\eA}_{j, \tau}),$$
where $\overline{\eA}_{j, \tau}$ denotes the complement of $\eA_{j, \tau}$. 
To prove the first part of the claim, it suffices to show that $\Pr(\run_{i,\tau} \mid \eA_{j,\tau}) \leq \Pr(\run_{i,\tau} \mid \overline{\eA}_{j,\tau})$. 
To see this, observe that when $v_i > v_j$, the availability of option $j$ at $\tau$ does not affect the event $\run_{i, \tau}$.
On the other hand, if $v_i \leq v_j$, then there is a non-negative probability that \Cref{alg:random-serve} selects option $j$ at $\tau$ (instead of option $i$) if option $j$ was available at $\tau$. 
In either case, we have $\Pr(\run_{i,\tau} \mid \eA_{j,\tau}) \leq \Pr(\run_{i,\tau} \mid \overline{\eA}_{j,\tau})$.
So, $\Pr(\run_{i,\tau} \mid \eA_{j,\tau}) \leq \Pr(\run_{i,\tau}) \leq  {x^*_{i,\tau}}/{2}$, by \eqref{eq:prob-run}. 
Finally, note that
the event $\run_{i,\tau}$ only depends on the execution of the algorithm until $\tau$.
Thus, for any $t > \tau$, we have $\Pr(\run_{i,\tau} \mid \eA_{j,t}) = \Pr(\run_{i,\tau} \mid \eA_{j,\tau}) \leq {x^*_{i,\tau}}/{2}$, proving the first part of the claim.
The proof of the second part of the claim is similar, and is omitted for brevity.
\end{proof}

We now use \Cref{lem:well-defined} to complete the proof.
To do so, we consider the quantity $1 - f_{j, t}$, which can be upper bounded by the probability of events where, either (i) the DM is busy at time $t$ because of another option, or (ii) option $j$ was considered at an earlier time $\tau < t$. 
By applying the union bound, we get
\begin{align*}
    1-f_{j,t} \,\, &\leq \,\, \sum_{\tau<t} \sum_{i\neq j} \Pr(\run_{i,\tau} \mid \eA_{j,t}) \cdot \Pr(S_i>t-\tau) + \sum_{\tau<t} \Pr(\mathtt{C}_{j,\tau} \mid \eA_{j,t})\\
    & \leq \,\, \sum_{\tau<t} \sum_{i \neq j} \frac{x^*_{i,\tau}}{2} \cdot \overline{F}_{S_i}(t-\tau) + \sum_{\tau < t} \frac{x^*_{j, \tau}}{2\cdot p_j(\tau)} \,\, \leq \,\, \frac{1}{2}\left(1 + \left(1 -\frac{x^*_{j,t}}{p_j(t)}\right)\right),
\end{align*}
where the second inequality uses \Cref{lem:well-defined} and the final inequality follows from the feasibility of $\x^*$. 
Rearranging the inequality completes the argument.

Next, we prove that \Cref{alg:random-serve} achieves an expected value of at least $\frac{1}{2} \cdot \left(1 - \frac{1}{e} \right)$ times the value of the optimal LP solution $\x^*$.
Let $X_t$ be a random variable that denotes the value selected by \Cref{alg:random-serve} at time $t$, and let $\algval = \sum_{t \geq 1}\E[X_t]$. 
We will set $X_t = 0$ if the DM is busy at time $t$ (that is, no option is selected at time $t$). 
Our key lemma establishes that $\E[X_t]\geq \frac{1}{2} \cdot \left(1 - \frac{1}{e} \right) \cdot\sum_{j=1}^n v_j \cdot x_{j,t}^*$ for all $t \geq 1$.
\begin{lemma}\label{lem:key}
        The expected value collected by \Cref{alg:random-serve} at time $t$ can be bounded as follows.
        \[
        \E[X_t] \,\, \geq \,\, \frac{1}{2}\cdot \left(1 - \frac{1}{e}\right)\cdot \sum_{j=1}^n v_j \cdot x^*_{j, t}.
        \]
    \end{lemma}
    We note that this immediately completes the proof of \Cref{thm:appx-bound} by summing over $t$. 
    \ignore{get 
    \[
    \mathtt{ALG} \,\, = \,\, \sum_{t \geq 1} \E[X_t] \,\, \geq \,\, \left(1 - \frac{1}{e}\right)\cdot \sum_{t \geq 1}\sum_{j=1}^n \frac{v_j \cdot x^*_{j, t}}{2 \cdot \Pr(\SF_t)} \geq \frac{1}{2} \cdot \left(1 - \frac{1}{e}\right) \sum_{t=1}^{nT}\sum_{j=1}^n v_j \cdot x^*_{j, t}.
    \]
    Since \ref{lp:LP1} upper bounds the optimal value, this completes the proof of \Cref{thm:appx-bound}.} \hfill \qedsymbol

\subsubsection{Proof of \Cref{lem:key}}
We finish up this subsection with a proof of \Cref{lem:key}.
We begin by establishing that, at every time step $t \geq 1$, the DM is free with probability at least $\frac{1}{2}$. This is formalized in the following claim.
    
\begin{claim}\label{prop:SF_t>1/2}
    For each $t\geq 1$, $\Pr(\SF_t)\geq 1/2$.
\end{claim}
\begin{proof}
    The result is clearly true for $t = 1$ since the DM is free at the start of the process. For $t > 1$, we analyze the probability that the DM is busy at time $t$.
    The DM can be busy at time $t$ only if it began evaluating some option $i \in [n]$ at an earlier time $\tau < t$ and the evaluation has not yet completed by time $t$. Using the union bound, we have:
    \begin{align*}
        1- \Pr(\SF_t) \,\, & \leq \,\, \sum_{\tau < t} \sum_{j=1}^n \Pr(\run_{j,\tau}) \cdot \overline{F}_{S_j}(t-\tau) \,\, \leq \,\, \sum_{\tau < t}\sum_{j=1}^n \frac{x_{j,\tau}^*}{2} \cdot\overline{F}_{S_j}(t-\tau) \,\, \leq \,\, \frac{1}{2}.
    \end{align*}
    Recall that $\run_{j,t}$ denotes the event that \Cref{alg:random-serve} selects option $j$ at time $t$. So, the second inequality follows from \eqref{eq:prob-run}, and the last inequality follows from the feasibility of $\x^*$ for \ref{lp:LP1}. 
\end{proof}

We are now ready to bound the expected value obtained by \Cref{alg:random-serve} at time $t$.
Recall that we set $X_t = 0$ if the DM is busy at time $t$ (that is, no option is selected at time $t$), and that we want to show 
$\E[X_t]\geq \frac{1}{2} \cdot \left(1 - \frac{1}{e} \right) \cdot\sum_{j=1}^n v_j x_{j,t}^*$ for all $t \geq 1$.
Without loss of generality, suppose that $v_1\geq v_2 \geq \cdots \geq v_n$. Then, $X_t=v_j$ if (i) the DM is free at time $t$, (ii) option $j$ is considered at time $t$ and, (iii) no option $i$ such that $i < j$ is considered at time $t$. 
Recall that $\mathtt{C}_{j,t}$ refers to the event that option $j$ is considered by $\ALG$ at time $t$ and denote $\overline{\mathtt{C}}_{j,t}$ as its complement event. Then, 
    \begin{align*}
        \Pr(X_t=v_j) \,\, &= \,\, \Pr(\SF_{t} \cap \mathtt{C}_{j,t} \cap \{ \forall i < j: \overline{\mathtt{C}}_{i,t} \})
         \,\, = \,\, \Pr(\SF_t) \cdot \Pr(\mathtt{C}_{j,t} \cap \{ \forall i < j: \overline{\mathtt{C}}_{i,t} \} \mid \SF_{t})\\
        & = \,\, \Pr(\SF_t) \cdot \Pr(\mathtt{C}_{j,t}\mid \SF_{t}) \cdot \prod_{i<j}\left(1-\Pr(\mathtt{C}_{i,t}\mid \SF_t)\right) \,\, = \,\, \Pr(\SF_t) \cdot \alpha_{j,t} \cdot\prod_{i< j} (1-\alpha_{i,t}),
    \end{align*}
    where the penultimate equality follows from the independence of events $C_{j, t}$ for all $j \in [n]$, and the final equality by defining $\alpha_{j, t} = \Pr(\mathtt{C}_{j,t}\mid \SF_t)$. 
    We proceed to use this notation to complete the proof of \Cref{lem:key}.
    
    \begin{proof}[Proof of \Cref{lem:key}]
    We begin by simplifying $\alpha_{j, t}$.
    Observe that by conditioning on the events $\eA_{j, t}$ and $\NC_{j, t}$, we get
    \begin{align*}
    \textstyle\alpha_{j,t} \,\, 
    &= \,\, \Pr(\mathtt{C}_{j,t} \mid \SF_t, \eA_{j,t}, \NC_{j,t}) \cdot \Pr(\NC_{j,t} \mid \SF_t, \eA_{j,t}) \cdot \Pr(\eA_{j,t}) \\
    &= \,\, \frac{x^*_{j,t}}{2\cdot f_{j,t}\cdot \Pr(\eA_{j,t})} \cdot \frac{\Pr(\NC_{j,t} , \SF_t, \eA_{j,t})}{\Pr(\SF_t, \eA_{j,t})} \cdot \Pr(\eA_{j,t}) \,\, = \,\, \frac{x^*_{j,t}}{2\cdot f_{j,t}} \cdot \frac{\Pr(\NC_{j,t} , \SF_t, \eA_{j,t})}{\Pr(\eA_{j,t})\cdot \Pr(\SF_t)} \\
    & = \frac{x^*_{j,t}}{2\cdot f_{j,t}} \cdot \frac{\Pr(\NC_{j,t} , \SF_t \mid \eA_{j,t})}{\Pr(\SF_t)} \,\, = \,\, \frac{x^*_{j,t}}{2 \cdot \Pr(\SF_t)},
    \end{align*}
    where the third equality uses the independence of $\eA_{j, t}$ and $\SF_t$.

    Furthermore, by the assumption that $v_1 \geq v_2 \geq  \cdots \geq v_n\geq 0$, there exist non-negative values $u_1, \ldots, u_n\geq 0$ such that we can write $v_j = \sum_{\ell = j}^n u_{\ell}$ for all $j \in [n]$. Using $X_t$ to denote the value of the highest valued available option at time $t$, we have:
    \begin{align*}
    \frac{\E[X_t]}{\sum_{j=1}^{n} \alpha_{j,t} v_j} \,\, &= \,\, \frac{\sum_{j=1}^{n} \alpha_{j,t} \prod_{i < j} (1 - \alpha_{i,t})  \cdot v_j} {\sum_{i=j}^{n} \alpha_{j,t} v_j} \,\, = \,\, \frac{\sum_{j=1}^{n} \alpha_{j,t} \prod_{i < j} (1-\alpha_{i,t})  \sum_{\ell=j}^{n} u_{\ell}}{\sum_{i=j}^{n} \alpha_{j,t} \sum_{\ell=j}^{n} u_{\ell}} \\
    &= \,\, \frac{\sum_{\ell=1}^{n} u_{\ell} \sum_{j=1}^{\ell} \alpha_{j,t} \prod_{i < j} (1-\alpha_{i,t}) }{\sum_{\ell=1}^{n} u_{\ell} \sum_{j=1}^{\ell} \alpha_{j,t}} \,\, \geq \,\, \min_{\ell: u_{\ell} \neq 0} \frac{u_\ell \sum_{j=1}^{\ell} \alpha_{j,t} \prod_{i < j} (1-\alpha_{i,t}) }{u_{\ell} \sum_{j=1}^{\ell} \alpha_{j,t}}\\
    & \geq \,\, \min_{\ell: u_{\ell} \neq 0} \frac{1- e^{-\sum_{i=1}^{\ell}\alpha_{i,t}}}{\sum_{j=1}^{\ell} \alpha_{j,t}} \,\, = \,\, \frac{1-e^{-\alpha_t}}{\alpha_t} \geq 1-\frac{1}{e},
    \end{align*}
    where the first inequality follows from $1+x \leq e^x$ for all $x$, the penultimate equality
    follows from the the fact that $f(x) = (1-e^{-x})/x$ is decreasing in $[0, 1]$, and final inequality comes from $\alpha_t \leq 1$.  
    Thus, we can conclude that 
    \[
    \E[X_t] = \sum_{j=1}^n v_j \alpha_{j,t}\prod_{i<j}(1-\alpha_{i,t}) \geq (1-1/e) \cdot \sum_{j=1}^n\alpha_{j,t} v_j = (1-1/e)\sum_{j=1}^n x_{j,t}^* \cdot v_{j} /(2\Pr(\SF_t)).
    \]
    \end{proof}

\subsection{Limitations of Our Algorithm and Analysis}\label{sec:limitations}
In this section, we present two results that
highlight the limitations of our algorithm and LP-based analysis. 
First, we show in \Cref{lem:approx_hard} that our approximation guarantee is nearly tight: we construct an instance in which $\ALG$ achieves at most a $(1 - 1/\sqrt{e}) \approx 0.393$ fraction of the LP optimum, demonstrating that our analysis cannot be significantly improved in general.
Second, in \Cref{lem:lp-gap}, we provide a bound on the integrality gap of \ref{lp:LP1}, establishing a fundamental limit on what any LP-based algorithm that uses \ref{lp:LP1} can achieve.

\begin{lemma}\label{lem:approx_hard}
    There exists a sequence of instances such that, as the number of options $n \to \infty$, the expected value obtained by \Cref{alg:random-serve} is at most $\left(1 - \frac{1}{\sqrt{e}}\right)$ times the value of \ref{lp:LP1}.
\end{lemma}
\begin{proof}
Consider the following instance of \ssx with $n$ identical options. Each option $j \in [n]$ has value $v_j = 1$, expiration time $E_j = 2$ with probability $1$, and service time $S_j = 1$ with probability $1$. 
Observe that setting $x^*_{j,1} = 1/n$ for all $j \in [n]$ is a feasible (and optimal) solution to~\ref{lp:LP1} with total objective value equal to $1$.

Since all options expire at time $2$, \Cref{alg:random-serve} can only collect value at time $t = 1$. At this time, the DM is free with probability $1$, and both $p_j(1)$ and $f_{j,1}$ equal $1$ for all $j \in [n]$. Therefore, each option $j$ is added to the consideration set independently with probability $x^*_{j,1} / (2 \cdot p_j(1) \cdot f_{j,1}) = 1/(2n)$.
Let $\mathcal{C}^*(1)$ denote the consideration set at time $1$. The probability that at least one option is added to $\mathcal{C}^*(1)$ is:
\[
\Pr(\mathcal{C}^*(1) \neq \emptyset) = 1 - \left(1 - \frac{1}{2n}\right)^n \xrightarrow[n \to \infty]{} 1 - \frac{1}{\sqrt{e}}.
\]
Since the algorithm selects an option and gains value only if $\mathcal{C}^*(1)$ is non-empty, the expected value obtained by \Cref{alg:random-serve} satisfies:
\(\ALG \to 1 - \frac{1}{\sqrt{e}},\) completing the proof.
\end{proof}
 
\begin{lemma}\label{lem:lp-gap}
    The integrality gap of \ref{lp:LP1} is at least $\left(1 - \frac{1}{e}\right)$. Formally, there exists a family of instances such that, as the instance size grows, the optimal value is at most $\left(1 - \frac{1}{e}\right)$ times the value of \ref{lp:LP1}.
\end{lemma}
\begin{proof}
Let $n$ be sufficiently large, and let $m < n$ be a fixed large constant such that $m$ divides $n$.
We consider the following instance: there are $n$ options partitioned into $m$ distinct types, with each type containing exactly $n/m$ options. All options have equal value, with $v_j = 1$ for all $j \in [n]$. For each option $j$ of type $k \in [m]$, we have:
\begin{equation*}
    S_{j} = 2^{k-1} \text{ with probability } 1\hspace{0.1cm} \text{ and } \hspace{0.1cm}
  E_{j} =
    \begin{cases}
      2 & \text{with probability $1-m/n$}\\
      2^{k-1}+2 & \text{with probability $m/n$}\\
    \end{cases}.
\end{equation*}
Consider the following solution: for options $j$ of type $k$, set $x_{j,2^{k-1}} = m/n$. 
It is easy to verify that this is a feasible solution to \ref{lp:LP1} and has value $m$.
Therefore, the value of~\ref{lp:LP1} is at least $m$.

We now compute the optimal value, denoted  $\OPT$, for this instance. 
By construction, each option of type $k$ expires either at time $t = 2$ or at time $t = 2^{k-1} + 2$. Therefore, after time $1$, the DM knows exactly which options have expired and which remain available for future selection. Moreover, from each type, at most two options can be selected: one before expiration at $t=1$, and potentially one more at its later expiration time. 
Since options of lower type have smaller processing times and earlier expiration, the optimal policy is to select options in increasing order of type. Specifically:
\begin{itemize}
    \item At time $t = 1$, the DM selects an option of type $1$.
    \item At time $t = 2$, with probability $1 - (1 - \frac{m}{n})^{n/m - 1}$, at least one other option of type $1$ remains.
    \item For each type $k = 2, \dots, m$, with probability $1 - (1 - \frac{m}{n})^{n/m}$, at least one option of that type remains available beyond time $t=2$.
    So, at time $t = 2^{k-1}+1$, an option of type $k$ will be selected.
\end{itemize}
Hence, the expected total reward is:
\[
\OPT = 1 + \left( 1- \left( 1- \frac{m}{n} \right)^{n/m-1} \right)+ (m-1)\left( 1- \left( 1- \frac{m}{n} \right)^{n/m} \right)
\]
Taking the limit as $n \to \infty$, we obtain: $\OPT \to 1+m(1-1/e)$. 
Combined with the fact that the value of~\ref{lp:LP1} is at least $m$, and letting $m \to \infty$, we conclude that the integrality gap approaches $1 - 1/e$, which completes the proof.
\end{proof}

\subsection{Constrained Sequential Selection with Expiration} \label{sec:extension}
We now describe the changes needed to handle the constrained versions of \ssx. 
We focus on two types of constraints: (1) \emph{deadlines}, where each option must be completed by a specified time to obtain value , and (2) a \emph{global knapsack constraint}, where the total weight of selected options must not exceed a given capacity. 
The main change required to handle these variants lies in modifying the LP that upper bounds the performance of the optimal online policy.
Specifically, we augment the LP with additional constraints to reflect feasibility under deadlines or knapsack limits, while maintaining the structure necessary for algorithmic rounding.

We begin by considering the first extension, in which each option is associated with a deadline.
Recall that in this setting, each option $j \in [n]$ is associated with a deadline $D_j$, and its value is obtained only if the evaluation of the option completes by time $D_j$. Letting $D_{\max} = \max_{j \in [n]} D_j$, the LP is as follows:
\ignore{\begin{align}
\text{maximize} \quad &\sum_{j=1}^n \sum_{t= 1}^{D_{j}}  \Pr(S_j \leq D_j -t) x_{j, t} v_j & \tag{\lpddl}\label{lp:LP2} \\ 
\text{subject to} \quad &\sum_{t = 1}^{D_j} \frac{x_{j, t}}{p_j(t)}\leq 1 & \forall j \in [n] \tag{\lpddla} \label{eq:LP2(a)} \\
 \quad & \sum_{j=1}^n \sum_{\tau=1}^{\min\{ t,D_j \}}  x_{j, \tau} \overline{F}_{S_j}(t-\tau) \leq 1  &  \forall t \in [1,\max_{j\in [n]} D_j] \tag{\lpddlb} \label{eq:LP2(b)}\\
 \quad &x_{j, t} \geq 0 & \forall j \in [n], t \in [1,D_{j}] \notag
\end{align}
}
\begin{align}
    \text{maximize} \quad & \textstyle \sum_{j=1}^n \sum_{t=1}^{D_j} \Pr(S_j \leq D_j - t)\cdot v_j \cdot x_{j,t}   \notag \\ 
     \text{s.t.} \ \ \forall j \in [n] \quad & \textstyle \sum_{t= 1}^{D_j} \frac{x_{j, t}}{p_j(t)}\leq 1  \tag{\lpddl}\label{lp:LP2} \\
     \forall t \in [D_{\max}] \quad & \textstyle \sum_{j=1}^n\sum_{\tau \leq \min\{t, D_j\}}  \overline{F}_{S_j}(t-\tau) \cdot  x_{j, \tau} \leq 1   \notag \\
      \forall j \in [n], \forall t \in [D_j] \quad &   x_{j, t} \geq 0.  \notag
\end{align}
The LP has a total of $O(n D_{\max})$ variables and constraints, and is therefore solvable in polynomial time provided that $D_{\max}$ is polynomial in $n$.
Similar to \Cref{thm:main-upperbound-opt}, we can show that \ref{lp:LP2} provides an upper bound on the optimum. 
The proof is similar to Theorem~\ref{thm:main-upperbound-opt}; we omit it for brevity.

\begin{theorem}\label{thm:main-dl-upperbound-opt}
    Given an instance of sequential selection with expiration with deadlines, \ref{lp:LP2} upper bounds the expected value of the optimal online policy. 
\end{theorem}

We apply \Cref{alg:random-serve} to this variant, and using an analysis identical to that in \Cref{sec:approx-alg}, we obtain the following.
\begin{theorem}\label{thm:appx-bound-deadlines}
    \Cref{alg:random-serve} achieves a $\textstyle \frac{1}{2}\cdot \left(1-\frac{1}{e}\right)$-approximation for \ssx with deadlines.
\end{theorem}

Our second extension considers a global knapsack constraint, where the total weight of selected options must not exceed a given capacity. Recall that in this setting, each option $j$ has a weight $w_j\geq 0$, and the total weight of selected options cannot exceed a given capacity $W$. 
Similar to \ssx with deadlines, we show that the results developed for the vanilla \ssx model extend naturally to this variant.
In particular, we modify the LP relaxation \ref{lp:LP1} by adding the following knapsack constraint:
\[
    \textstyle \sum_{t=1}^{nT}\sum_{j=1}^n w_j \cdot x_{j, t} \leq W.
\]
The resulting LP continues to serve as a valid upper bound on the expected value of the optimal online policy.
Again, we apply \Cref{alg:random-serve} to this variant. 
The analysis is identical except one change.
We declare the algorithm a ``failure'' (and collect no value) if the total weight of the selected options exceeds $W$. 
From \Cref{eq:prob-run}, the probability that option $j$ is selected at time step $t$ is at most $x^*_{j, t}/2$. Thus, the expected weight of selected options in an independent execution of \Cref{alg:random-serve} is at most $\sum_{t \geq 1}\sum_{j=1}^n w_j \cdot (x^*_{j, t})/2 \leq W/2$. 
By a Chernoff bound, the probability of failure is at most $e^{-\Omega(W^2/nw_{\max})}$ where $w_{\max} = \max_j w_j$.
The main theorem is as follows.
\begin{theorem}\label{thm:appx-bound-knapsack}
    \Cref{alg:random-serve} achieves a $\textstyle (1 - e^{-\Omega(W^2/nw_{\max})}) \cdot \frac{1}{2}\cdot \left(1-\frac{1}{e}\right)$-approximation for \ssx with a knapsack constraint.
\end{theorem}

As with \Cref{thm:appx-bound}, both of the above results assume access to exact $f_{j, t}$ values.
However, using a similar sampling-based analysis as described in \Cref{app:sampling} establishes \Cref{thm:main-deadlines} and \Cref{thm:main-knapsack}.


\section{Greedy Policies for I.I.D. Evaluation Times}{\label{greedy_section}}

In this section, we prove Theorem~\ref{thm:greedy} and show that the greedy policy, denoted henceforth as $\ALG_{Greedy}$, that always chooses to evaluate the available option of largest value guarantees at least a $1/2$-approximation relative to the online optimum when the evaluations times are independent and identically distributed. 
We define some notation to simplify the exposition in this section. We use $\OPT$ and $v(\OPT)$ to denote the optimal policy and its expected value in this section. In general, we use $v(\cdot)$ to denote the expected value of a given policy.

The formal argument for our proof requires using the full framework of MDPs, and we defer these details to Appendix~\ref{app:greedy_analysis}. 
Here, we give a intuitive sketch of our proof, which roughly involves two steps.
\begin{enumerate}
    \item We begin by iteratively modifying $\OPT$ and deriving an intermediate policy that behaves ``almost'' like $\ALG_{Greedy}$ in every step and has expected value that is at least half of $v(\OPT)$.

    \item Then, we show that $\ALG_{Greedy}$ outperforms the intermediate policy developed in Step 1; thereby proving the result.
\end{enumerate}

\paragraph{Step 1: Designing the intermediate policy.} We iteratively construct a sequence of policies $\ALG^0, \ldots, \ALG^t, \ldots$, starting with $\ALG^0 = \OPT$, where each $\ALG^{t}$ for $t \ge 1$ is obtained from $\ALG^{t-1}$ by making $\ALG^{t-1}$ behave greedily at time $t$. The construction is intuitive; however, a formal argument is more involved and is deferred to Appendix~\ref{app:greedy_analysis}.

We now describe how we construct $\ALG^t$ for $t \geq 1$ given $\ALG^{t-1}$. 
The idea is simple:
(i) Follow $\ALG^{t-1}$ as is for every time $\tau \neq t$. (ii) For time $\tau = t$, if the DM is free, choose the available option of largest value, say $j^*$. (iii) If for some time $\tau > t$, $\ALG^{t-1}$ evaluates option $j^*$, then, $\ALG^t$ will \emph{simulate} evaluating it by taking a random evaluation time $S_{j^*}$.
Clearly, policy $\ALG^t$ is well-defined. Furthermore, the following result quantifies the difference in values for $\ALG^t$ and $\OPT$.
\begin{lemma}\label{fact:general_alg_t}
    Let $v(\ALG^t[t+1])$ denote the expected value obtained by $\ALG^t$ from time $t+1$ onward.
    Then, for any $t\geq 1$, we have
\begin{equation*}\label{greedy_iid_punchline}
    2\cdot \mathbb{E}\left[\sum_{\tau\leq t} v_{i_\tau} \right] + v(\ALG^t[t+1]) \,\, \geq \,\, v(\OPT),
\end{equation*}
where $i_1,\ldots,i_t$ are the indices of options being evaluated by $\ALG^t$.
\end{lemma}
The lemma states that twice the total value of the options chosen up to time $t$, plus the value obtained from time $t+1$ onward by $\ALG^t$, provides an upper bound on the optimal value. 
This is easy to see for $t=1$, since in this case $j^*$ is the index of the highest-valued option, and we have
\begin{align}
v(\OPT) \le v_{j_1}+ v(\ALG^1[2]) + v_{j^*}\Pr(\OPT \text{ chooses $j^*$ at some point}) \le v(\ALG^1[2]) + 2\cdot v_{j^*}, \label{eq:bound_alg1}
\end{align}
where $j_1$ denotes the first item selected by $\OPT$.
For $t \ge 2$, however, $j^*$ in the definition of $\ALG^t$ depends on the history. Hence, a formal proof of \Cref{fact:general_alg_t} requires carefully defining the state and action spaces for policies and explicitly describing how policy $\ALG^{t}$ is constructed from $\ALG^{t-1}$ through these definitions. The proof also relies on coupling the decision times of $\OPT$ and $\ALG^{t}$, which is possible because the evaluation times are i.i.d. The formal construction of $\ALG^t$ and the proof of \Cref{fact:general_alg_t} are presented in Appendix~\ref{app:greedy_analysis}.
However, as a consequence of \Cref{fact:general_alg_t}, we can define the intermediate policy $\ALG'=\ALG^{T}$ and conclude that $v(\ALG')\geq \frac{1}{2}\cdot v(\OPT)$.

\paragraph{Step 2: Comparing the intermediate policy with $\ALG_{Greedy}$.}
By construction, 
the intermediate policy $\ALG'$ behaves greedily when selecting an option.
Moreover, whenever the intermediate policy is simulating an option, it receives no value. 
In contrast, $\ALG_{Greedy}$ always chooses an available option of largest value whenever possible. 
Therefore, it follows immediately that $v(\ALG_{Greedy}) \ge v(\ALG')$. Combining this with the previous bound completes the proof of \Cref{thm:greedy}. 


\newcommand{\conset}
{\textsc{ConSet}\xspace}
\newcommand{\safe}{\textsc{Safe}\xspace}
\newcommand{\greedy}{\textsc{Greedy}\xspace}
\newcommand{\ouralg}{\textsc{SimAlg}\xspace}
\newcommand{\ur}{\textsc{UR}\xspace}
\newcommand{\batch}{\textsc{ENS}\xspace}

\section{Computational Results}\label{sec:num_exp}
In this section, we provide a summary of our computational experiments. 
We consider two classes of instances: one constructed using real-world data on the performance and cost of large language models (LLMs), and another based on a classical call center scheduling dataset.
Before presenting our empirical findings, we begin by detailing the various algorithms we compare in our experiments.

\medskip
\noindent \textbf{Algorithms.}
We run five different algorithms on our computational setup. 
The first algorithm is our LP-based algorithm executed using simulations: we denote it as \ouralg  (see \Cref{sec:approx-alg} for the details). 
As computing the $f_{j,t}$ values require simulations, we test a heuristic variant, denoted \conset, 
that includes each option $j$ into the consideration set with probability $x^*_{j,t}/[p_j(t)\cdot(1-\sum_{\tau<t}x^*_{j,\tau}/p_j(\tau))]$ (the scaling ensures that the probabilities are valid).
While our theoretical analysis requires options to be only considered once, we implement both \ouralg and \conset by allowing reconsideration of previously considered options. 
This is more natural in practical settings and improves empirical performance.

As baselines, we evaluate two additional methods.
The first is \safe, an LP-based sampling algorithm originally proposed in the context of online matching~\citep{dickerson2021allocation}. 
The second is a natural greedy algorithm that always selects the highest-value available option whenever the DM is free.
Finally, we compare against a variant of the efficient nonmyopic search (\batch) heuristic from \citet{jiang2019cost}, tailored to our setting as described in \Cref{app:comp-sec}. 
At a high level, this algorithm estimates the marginal value of selecting an option by combining immediate reward and an approximate look-ahead value based on residual budget and job availability.

\smallskip

We conducted all of our experiments using Python~3.7.6 with Gurobi~10.0.3 and executed on a 2.4\,GHz Quad-Core Intel Core~i5 with 8\,GB 2.4\,MHz LPDDR3 memory.

\subsection{Active Search Instance}\label{sec:active-learning}

\paragraph{Instances.} 
In order to construct our test instances, we first construct a dataset from a public Large Language Model (LLM) leaderboard reporting performance on the GPQA benchmark.\footnote{This is a highly complex benchmark that evaluates reasoning quality and reliability across biology, physics, and chemistry. Data accessed October 12, 2025. \url{https://www.vellum.ai/llm-leaderboard}.} 
The dataset comprises 26 models, with an entry for each model capturing key features like token processing times, latency, accuracy, etc. 

Using this dataset, we create an instance of \ssx as follows. 
We begin by selecting $\ell \in \{1, \ldots, 10\}$ and create an instance that has $\ell$ options corresponding to each model for a total of $26\ell$ options.
For each option $ j$, the corresponding GPQA accuracy is recorded as its value.
We set the service time of each option by combining the token-processing time and latency of the underlying model as follows.
We assume a workload of $2{,}000$ tokens per request, so token-processing time is $T_{\text{tok}}=2000/v$ for speed $v$ (tokens/s), with empirical summary $\mu_s=21$ seconds, while latency ($T^j_{\text{lat}}$) for each model has $\mu_\ell=7.4$ seconds. 
For each model $j$, we fit Gaussians $\mathcal{N}(T^j_{\text{tok}},\sigma_s^2)$ and $L\!\sim\!\mathcal{N}(T^j_{\text{lat}},\sigma_\ell^2)$, then discretize each to a positive integer-valued PMF via truncation and normalization of weights $w(k)\!\propto\!\exp\!\big(-\tfrac12((k-\mu)/\sigma)^2\big)$. 
This is a standard technique used for estimating such quantities; see for example, \citet{canonne2020discrete}, \citet{ELAFFENDI199235}, and \citet{huggins2010inventory}.

Thus, for each model $j$, using constructed PMF $p_j^{\text{tok}}$ and $p_j^{\text{lat}}$ and assuming independence, we obtain the service-time distribution as their discrete convolution $p_j^{\text{svc}} = p_j^{\text{tok}} * p_j^{\text{lat}}$. 
This gives us an integer-valued PMF for the service time of each option that integrates both compute speed and latency. 
To model option expiration, we assign \emph{departure} probabilities $q_1, \ldots, q_n$ that are proportional to input cost per token $c_j$. 
This reflects the premise that, under constrained budgets, higher-cost models are less likely to be regarded as viable options. 
However, the range of the costs is prohibitively large (ranging from $\$0.07$ to $\$75$), and we find it convenient to work with a notion of ``log-normalized'' costs.
To do this, we define the log-normalized score of an option $j$ as
\[
s_j = \frac{\log c_j - \min_k \log c_k}{\max_k \log c_k - \min_k \log c_k}.
\]
Observe that $s_j \in [0, 1]$. We then map $s_j$ to a bounded departure probability via
\[
q_j = q_{\min} + (q_{\max}-q_{\min})\, s_j \quad \text{with} \quad 0<q_{\min}<q_{\max}<1.
\]
where the bounds $q_{\min}$ and $q_{\max}$ are used to avoid boundary probabilities $0$ and $1$. Given $q_j$, we consider two departure schemes. In Scheme I (\textsc{S1}), option $j$ departs independently in each period with probability $q_j$, so its survival probability over a horizon $T$ is $(1-q_j)^T$.
In Scheme II (\textsc{S2}), $q_j$ is interpreted as the probability that option $j$ departs during the full time horizon; thus, the per-period departure probability $\tilde q_j$ is chosen so that $1-(1-\tilde q_j)^T = q_j$. 

We consider three value regimes using results from LLM instance. 
In the first regime (\textsc{VR1}), we directly adopt the instance generation procedure and set the parameters as follows: $\sigma_s = 1$, $\sigma_l = 0.5$, $q_{\min} = 0.1$, and $q_{\max} = 0.95$. 
In a second variant (\textsc{VR2}), we construct values that are positively correlated with service time by scaling the service time and adding small Gaussian perturbations to induce natural variability. 
In the third (\textsc{VR3}), we correlate values with departure probabilities so that lower-valued options have smaller departure probabilities and thus remain in the system longer. 

\paragraph{Results.} For every instance, we compute the final value obtained by taking an average over $100$ independent realizations.
We normalize the value obtained by each algorithm by the LP benchmark value, and report how much of the LP value is obtained by the tested algorithms. Note that across all value regimes, the LP objective remains a valid benchmark, since item values are deterministic and do not depend on the evaluation or starting time of an option. 
Specifically, for an algorithm $\mathcal{A}$, we report:
\[
\frac{\text{average value obtained by }\mathcal{A}}{\text{value of LP benchmark}} 
\]
We begin by presenting results for the \textsc{S1} setting in Table~\ref{tab1}, which corresponds to a regime where most items depart by the end of time horizon.
This setting is unconstrained, with no deadline or cardinality constraints. Instances are indexed by their value regime and the  size parameter $\ell$. 

\begin{table}[H]
\centering
\begin{tabular}{c|c|c|c|c|c|c|}
\toprule
Instance & \textsc{SimAlg}\xspace & \textsc{ConSet}\xspace & \textsc{Safe}\xspace & \textsc{Greedy}\xspace & \textsc{\batch}\xspace \\
\midrule
$\textsc{VR1}-1$  & 0.654 (0.059) & 0.701 (0.038) & 0.806 (0.027) & 0.721 (0.047) & 0.435 (0.000) \\
$\textsc{VR1}-2$  & 0.714 (0.039) & 0.705 (0.028) & 0.776 (0.036) & 0.571 (0.027) & 0.517 (0.028) \\
$\textsc{VR1}-3$  & 0.690 (0.066) & 0.681 (0.053) & 0.799 (0.046) & 0.578 (0.020) & 0.601 (0.014) \\
$\textsc{VR1}-4$  & 0.677 (0.038) & 0.693 (0.041) & 0.818 (0.041) & 0.577 (0.022) & 0.593 (0.023) \\
$\textsc{VR1}-5$  & 0.664 (0.030) & 0.747 (0.031) & 0.739 (0.039) & 0.533 (0.011) & 0.582 (0.016) \\
$\textsc{VR1}-6$  & 0.665 (0.030) & 0.707 (0.029) & 0.815 (0.019) & 0.579 (0.019) & 0.622 (0.023) \\
$\textsc{VR1}-7$  & 0.714 (0.018) & 0.731 (0.039) & 0.796 (0.027) & 0.402 (0.007) & 0.611 (0.012) \\
$\textsc{VR1}-8$  & 0.607 (0.025) & 0.704 (0.037) & 0.728 (0.020) & 0.491 (0.013) & 0.525 (0.015) \\
$\textsc{VR1}-9$  & 0.656 (0.027) & 0.692 (0.030) & 0.767 (0.022) & 0.443 (0.009) & 0.619 (0.013) \\
$\textsc{VR1}-10$ & 0.666 (0.032) & 0.679 (0.026) & 0.732 (0.023) & 0.467 (0.012) & 0.546 (0.013) \\

$\textsc{VR2}-1$ & 0.732 (0.077) & 0.698 (0.114) & 0.837 (0.038) & 0.930 (0.013) & 0.936 (0.015) \\
$\textsc{VR2}-2$ & 0.720 (0.033) & 0.775 (0.074) & 0.846 (0.024) & 0.788 (0.009) & 0.291 (0.011) \\
$\textsc{VR2}-3$ & 0.681 (0.085) & 0.645 (0.074) & 0.762 (0.053) & 0.819 (0.013) & 0.412 (0.009) \\
$\textsc{VR2}-4$ & 0.798 (0.038) & 0.809 (0.029) & 0.863 (0.020) & 0.725 (0.011) & 0.385 (0.008) \\
$\textsc{VR2}-5$ & 0.701 (0.039) & 0.654 (0.065) & 0.745 (0.032) & 0.697 (0.007) & 0.417 (0.009) \\
$\textsc{VR2}-6$ & 0.773 (0.012) & 0.827 (0.018) & 0.848 (0.012) & 0.722 (0.012) & 0.398 (0.011) \\
$\textsc{VR2}-7$ & 0.698 (0.027) & 0.770 (0.023) & 0.818 (0.021) & 0.463 (0.004) & 0.336 (0.003) \\
$\textsc{VR2}-8$ & 0.678 (0.032) & 0.749 (0.032) & 0.771 (0.018) & 0.463 (0.003) & 0.283 (0.003) \\
$\textsc{VR2}-9$ & 0.616 (0.024) & 0.697 (0.015) & 0.779 (0.019) & 0.457 (0.006) & 0.393 (0.005) \\
$\textsc{VR2}-10$ & 0.654 (0.026) & 0.691 (0.020) & 0.764 (0.021) & 0.473 (0.004) & 0.332 (0.007) \\

$\textsc{VR3}-1$ & 0.673 (0.049) & 0.661 (0.050) & 0.767 (0.058) & 0.485 (0.004) & 0.844 (0.017) \\
$\textsc{VR3}-2$ & 0.599 (0.034) & 0.677 (0.048) & 0.753 (0.038) & 0.484 (0.021) & 0.609 (0.028) \\
$\textsc{VR3}-3$ & 0.548 (0.061) & 0.643 (0.054) & 0.727 (0.039) & 0.439 (0.016) & 0.576 (0.029) \\
$\textsc{VR3}-4$ & 0.488 (0.027) & 0.652 (0.038) & 0.662 (0.045) & 0.446 (0.017) & 0.670 (0.013) \\
$\textsc{VR3}-5$ & 0.553 (0.046) & 0.685 (0.040) & 0.783 (0.036) & 0.454 (0.024) & 0.627 (0.016) \\
$\textsc{VR3}-6$ & 0.514 (0.035) & 0.681 (0.053) & 0.692 (0.037) & 0.474 (0.032) & 0.657 (0.021) \\
$\textsc{VR3}-7$ & 0.568 (0.034) & 0.668 (0.035) & 0.743 (0.032) & 0.297 (0.008) & 0.651 (0.014) \\
$\textsc{VR3}-8$ & 0.608 (0.034) & 0.653 (0.027) & 0.645 (0.038) & 0.319 (0.011) & 0.705 (0.012) \\
$\textsc{VR3}-9$ & 0.505 (0.026) & 0.688 (0.027) & 0.660 (0.026) & 0.316 (0.010) & 0.652 (0.011) \\
$\textsc{VR3}-10$ & 0.550 (0.024) & 0.655 (0.036) & 0.696 (0.039) & 0.317 (0.011) & 0.727 (0.009) \\
\bottomrule
\end{tabular}
\caption{Active search instance. Algorithm value with respect to the LP benchmark, reported as ratios. Results are rounded to retain three decimals for consistency. Sample variances are shown in parentheses.}
\label{tab1}
\end{table}

Table \ref{tab1} shows that the LP-based \ouralg\ consistently achieves strong approximation ratios across all value regimes and increasing instance sizes, with low variance. Compared to LP-based methods using a consideration set, whose performance may be affected by the stochastic inclusion of jobs in the set, \safe constitutes a stronger heuristic since it executes a job whenever the server is available, guided directly by the LP solution. In contrast, the performance of \greedy\ and \batch\ varies drastically as the number of items changes, indicating a lack of robustness to instance scale. 
We next present results for the \textsc{S2} setting in Table~\ref{tab2}, which corresponds to a more patient expiration regime; here, we fix $\ell = 10$. 
This setting is subject to deadline constraints $D \in \{40, 200, 600\}$. Instances are indexed by their value regime and the deadline parameter. 

\begin{table}[H]
\centering
\begin{tabular}{c|c|c|c|c|c|c|}
\toprule
Instance & \textsc{SimAlg}\xspace & \textsc{ConSet}\xspace & \textsc{Safe}\xspace & \textsc{Greedy}\xspace & \textsc{\batch}\xspace \\
\midrule
$\textsc{VR1}-40$ & 0.780 (0.002) & 0.725 (0.011) & 0.923 (0.003) & 0.272 (0.000) & 0.836 (0.006) \\
$\textsc{VR1}-200$ & 0.825 (0.002) & 0.723 (0.002) & 0.865 (0.001) & 0.391 (0.000) & 0.817 (0.000) \\
$\textsc{VR1}-600$ & 0.800 (0.001) & 0.681 (0.001) & 0.902 (0.000) & 0.504 (0.000) & 0.830 (0.000) \\
$\textsc{VR2}-40$ & 0.786 (0.013) & 0.723 (0.007) & 0.876 (0.003) & 0.443 (0.000) & 0.788 (0.002) \\
$\textsc{VR2}-200$ & 0.788 (0.001) & 0.692 (0.002) & 0.855 (0.002) & 0.588 (0.000) & 0.867 (0.001) \\
$\textsc{VR2}-600$ & 0.777 (0.000) & 0.671 (0.002) & 0.895 (0.001) & 0.582 (0.000) & 0.876 (0.000) \\
$\textsc{VR3}-40$ & 0.782 (0.006) & 0.694 (0.003) & 0.893 (0.003) & 0.272 (0.000) & 0.787 (0.000) \\
$\textsc{VR3}-200$ & 0.754 (0.002) & 0.709 (0.001) & 0.867 (0.002) & 0.367 (0.000) & 0.890 (0.001) \\
$\textsc{VR3}-600$ & 0.778 (0.001) & 0.697 (0.001) & 0.862 (0.001) & 0.473 (0.001) & 0.862 (0.003) \\
\bottomrule
\end{tabular}
\caption{Active search instance with deadlines. Algorithm value with respect to the LP benchmark, reported as ratios. Results are rounded to retain three decimals for consistency. Sample variances are shown in parentheses.}
\label{tab2}
\end{table}

Table \ref{tab2} highlights the effect of deadline on algorithm performance, with smaller deadlines corresponding to more constrained decision windows. Across all value regimes, the LP-based \ouralg\ exhibits stable approximation ratios, whereas \greedy\ is highly sensitive to tight deadlines and performs poorly when the decision window is short. \conset exhibits high sensitivity to deadline extensions, resulting in a performance downgrade as the time horizon increases. We finally present results for the \textsc{S2} in Table~\ref{tab3} setting under both deadline and cardinality constraints. We fix the deadline at $D = 400$ and consider cardinalities $W \in \{9, 16, 25\}$. Instances are indexed by their value regime and the cardinality parameter.

\begin{table}[H]
\centering
\begin{tabular}{c|c|c|c|c|c|c|}
\toprule
Instance & \textsc{SimAlg}\xspace & \textsc{ConSet}\xspace & \textsc{Safe}\xspace & \textsc{Greedy}\xspace & \textsc{\batch}\xspace \\
\midrule
$\textsc{VR1}-9$ & 0.867 (0.012) & 0.900 (0.009) & 0.944 (0.006) & 1.000 (0.000) & 0.798 (0.000) \\
$\textsc{VR1}-16$ & 0.825 (0.002) & 0.723 (0.002) & 0.865 (0.001) & 0.391 (0.000) & 0.817 (0.000) \\
$\textsc{VR1}-25$ & 0.800 (0.001) & 0.681 (0.001) & 0.902 (0.000) & 0.504 (0.000) & 0.830 (0.000) \\
$\textsc{VR2}-9$ & 0.748 (0.007) & 0.722 (0.007) & 0.918 (0.002) & 0.997 (0.000) & 0.186 (0.000) \\
$\textsc{VR2}-16$ & 0.798 (0.003) & 0.665 (0.006) & 0.881 (0.003) & 0.990 (0.001) & 0.197 (0.000) \\
$\textsc{VR2}-25$ & 0.762 (0.002) & 0.622 (0.002) & 0.878 (0.001) & 0.858 (0.001) & 0.354 (0.000) \\
$\textsc{VR3}-9$ & 0.811 (0.005) & 0.689 (0.004) & 0.977 (0.002) & 1.000 (0.000) & 0.536 (0.000) \\
$\textsc{VR3}-16$ & 0.781 (0.003) & 0.719 (0.003) & 0.912 (0.004) & 1.000 (0.000) & 0.510 (0.000) \\
$\textsc{VR3}-25$ & 0.712 (0.003) & 0.677 (0.007) & 0.877 (0.004) & 0.685 (0.000) & 0.578 (0.000) \\
\bottomrule
\end{tabular}
\caption{Active search instance with deadlines and cardinalities. Algorithm value with respect to the LP benchmark, reported as ratios. Results are rounded to retain three decimals for consistency. Sample variances are shown in parentheses.}
\label{tab3}
\end{table}

Table \ref{tab3} reports performance under simultaneous deadline and cardinality constraints. The LP-based \ouralg\ maintains strong and stable approximation ratios as the cardinality constraint is relaxed, demonstrating robustness to capacity variation. \safe continues to perform consistently well across the different value regimes. In contrast, the performance of \greedy\ and \batch\ varies drastically with the cardinality constraint across all value regimes.

\subsection{Call Center Instance}\label{sec:call-center}

\paragraph{Instances.} We use a public domain dataset from the call center of `Anonymous Bank' in Israel.\footnote{Data can be accessed from \url{https://seelab.net.technion.ac.il/data/}.}
For our experimental setup, we use data originating from January $1999$, which comprises $30,140$ call entries, as the training dataset. The call records are primarily split into three distinct categories: ``priority'', ``regular'', and ``new'', each consisting $16,977$, $4,440$ and $8,723$ entries, respectively. We assign values $8$, $2$, and $1$ to priority, regular, and new calls.
Furthermore, using $20$ seconds as a reasonable unit time-step, we extract distributional information for both the (stochastic) evaluation and the expiration time. In particular, the expiration time corresponds to the duration a caller is willing to wait before abandoning the call. 
We present these as empirical point mass function distributions in Figures~\ref{pic:service} and~\ref{pic:departure} in Appendix~\ref{app:comp-sec-cc}. 
These are then used to sample instances (in particular, the service times and expiration times for the options) for our computational setup. 
As a sanity check, we compute the standard root mean squared error (RMSE) of the learned expiration and evaluation point mass functions for January and February, and observe that the error is 
$0.013$ and $0.006$, respectively. 

We consider problem sizes \( n \in \{10, 20, \ldots, 50\} \).\footnote{Larger problem sizes are not considered because, under the call center dynamics, nearly all customers would abandon before service, yielding no additional insight.}
For each size, we generate ten instances following the caller category frequency distribution, assigning each customer corresponding evaluation and expiration time distributions.

\paragraph{Results.} Table~\ref{tab4} presents the value ratios along with the sample variance, where each ratio represents the average over $100$ independent runs obtained by comparing heuristic outcomes against the LP benchmark. 
Instances are indexed by problem size $n$.

\begin{table}[h]
\centering
\begin{tabular}{c|c|c|c|c|c|c|}
\toprule
Instance & \textsc{SimAlg}\xspace & \textsc{ConSet}\xspace & \textsc{Safe}\xspace & \textsc{Greedy}\xspace & \textsc{\batch}\xspace \\
\midrule
$\mathtt{Call}-10$ & 0.580 (0.065) & 0.562 (0.031) & 0.610 (0.042) & 0.685 (0.048) & 0.604 (0.031) \\
$\mathtt{Call}-20$ & 0.760 (0.139) & 0.744 (0.067) & 0.786 (0.089) & 0.907 (0.105) & 0.769 (0.066) \\
$\mathtt{Call}-30$ & 0.761 (0.158) & 0.749 (0.073) & 0.791 (0.097) & 0.924 (0.120) & 0.772 (0.072) \\
$\mathtt{Call}-40$ & 0.766 (0.187) & 0.754 (0.079) & 0.799 (0.104) & 0.948 (0.142) & 0.784 (0.093) \\
$\mathtt{Call}-50$ & 0.761 (0.192) & 0.765 (0.077) & 0.808 (0.104) & 0.939 (0.141) & 0.797 (0.084) \\
\bottomrule
\end{tabular}
\caption{Call center instance. Algorithm value with respect to the LP benchmark, reported as ratios. Results are rounded to retain three decimals for consistency. Sample variances are shown in parentheses.}
\label{tab4}
\end{table}

Across all instance sizes, the LP-based \ouralg\ maintains stable and competitive performance relative to the LP benchmark, exhibiting only moderate variance as uncertainty increases with $n$. In line with previous findings, \conset yields results comparable to \ouralg, with the added advantage of bypassing the simulations required to estimate consideration set inclusion probabilities. Because different customer categories have similar service-time and expiration-time distributions, the \greedy\ algorithm performs well in this setting by prioritizing customers with the largest value.

\subsection{Discussions}\label{sec:discussion}
Overall, \ouralg\ achieves consistently high approximation ratios across all value regimes in both the active search instances and the call center instances. 
While heuristic methods such as \safe\ and \conset\ exhibit competitive—and sometimes superior—empirical performance in specific regimes, establishing formal theoretical guarantees for these approaches remains an open problem. 
We also observe that \greedy\ performs particularly well in the \textsc{VR2} regime, where values are positively correlated with service times, and that \batch\ benefits substantially from the presence of deadline constraints. Consequently, although these heuristics serve as effective substitutes in certain settings, \ouralg\ remains the more theoretically grounded choice when provable guarantees are prioritized.


\section{Conclusion}\label{sec:conclusion}
In this paper, we introduced and studied the sequential selection with expirations (\ssx) problem, a general framework that models decision-making under uncertainty, 
especially focused on uncertain evaluation times and probabilistic expirations.
Even when evaluation times are deterministic, we show that the problem is NP-hard, motivating the design of efficient algorithms with provable approximation guarantees. 
Our primary contribution is the design of a polynomial-time approximation algorithm for the general setting using a time-indexed LP relaxation of any online policy. 
We complement this with a tight analysis, demonstrating the limitations of our algorithm and the LP relaxation. 
For the special case with i.i.d. service times, we show that a simple greedy policy achieves a $1/2$-approximation, which is provably the best possible among such policies. 
Our framework naturally extends to constrained variants, including deadline and knapsack constraints, where similar approximation guarantees continue to hold.
To demonstrate practical relevance, we presented computational experiments in two distinct application domains: active search using real LLM performance data, and a classical call center scheduling dataset. 
Across both settings, our algorithms are easy to implement and perform robustly, offering decision-makers a principled way to make selection decisions under uncertainty of evaluation times and the possibility of options expiring.
We conclude by highlighting several promising directions for future research and practically relevant extensions, as summarized below.

\paragraph{Arrival Processes.}
A natural and practically relevant extension of our model involves incorporating \emph{option arrivals} over time, where options are not all available at the beginning but instead become available over time.
This setting captures many real-world scenarios in which opportunities arise dynamically, such as candidate availability in recruiting pipelines, new product features being ready for testing, or customers entering a service queue.
One way to formalize arrivals is through \emph{option-dependent release times}, 
where each option becomes available at a predetermined or stochastic time. 
When release times are \emph{deterministic}, our LP-based framework can be extended naturally by only allowing LP variables for time periods after an option’s release. In this case, both our LP relaxation and rounding approach go through with minimal modifications.
We leave open the question of handling more general arrival processes, for example, random order arrivals or stochastic arrivals where at each time $t$, an option of type $k$ arrives with probability $p_{t, k}$.

\paragraph{Parallel Evaluation and Other Combinatorial Constraints.}
Several natural extensions of our model involve supporting multiple decision-makers (DMs) or more general combinatorial constraints on how options are assigned. These extensions are both theoretically interesting and practically motivated.
A particularly natural direction is the \emph{multi-DM setting}, where multiple agents can evaluate options in parallel. Our LP framework can be extended to this setting. Specifically, by indexing DMs using an index $i$ and introducing decision variables $x_{j,t,i}$ to represent the probability that DM $i$ selects option 
$j$ at time $t$, we can obtain a valid relaxation that upper bounds the performance of the optimal policy. The constraints in this extended LP mirror those of our original formulation (see~\ref{lp:LP1}), applied to each DM.
However, while the LP upper bound carries over, it remains unclear how to adapt our algorithmic framework to this multi-agent setting.
More broadly, our framework can potentially be adapted to incorporate other combinatorial feasibility constraints; for example, matroid, matching, or scheduling constraints, by designing appropriate LP relaxations.
Designing approximation algorithms for such extensions remains a promising direction for future work.

\paragraph{Better Approximations.}
As we show, the approximation ratio achievable by our LP-based methodology is inherently limited
$1-1/e$ (see \Cref{sec:limitations}). 
Recent work by \citet{segev2024near} presents a quasi-polynomial time approximation scheme for the special case where all service times are uniform (and deterministic) and expirations follow geometric distributions.
Achieving significantly better guarantees would thus require either a stronger relaxation or an entirely different algorithmic approach. 
Exploring new algorithmic techniques offers an exciting avenue to break this barrier. Whether such improvements are possible, either in the general model or in natural special cases, remains an important open question, as also highlighted by \citet{segev2024near}.

\bibliographystyle{plainnat}
\bibliography{ref}


\appendix

\section{NP-Hardness and Comparison to the Online Optimum}\label{appendix:benchmark}
Recall that our main result is a polynomial-time constant-factor approximation algorithm for \ssx relative to the optimal online algorithm (see \Cref{thm:main-approx}). 
In this section, we discuss why we design an approximation algorithm rather than aiming for an exact optimal solution and also explain our choice of benchmark. Specifically, we justify our use of the optimal online algorithm as a benchmark (as opposed to the offline optimum) and highlight the computational challenges that make exact optimization intractable in our setting.

\medskip
\noindent \textbf{NP-Hardness.} 
In an instance of the NP-hard knapsack problem, we are given a set $[n]$ of items with non-negative costs $\{c_j : j \in [n]\}$ and values $\{v_j : j \in [n]\}$, and a budget $C$ on the total cost. The goal is to select a subset of items of total cost at most $C$ that maximizes the total value.
We note that \ssx generalizes the deterministic knapsack problem, proving that \ssx is NP-hard. 
Suppose that for all options $j \in [n]$, the evaluation time $S_j = c_j$ with probability $1$.
Furthermore, suppose that the probability that $j$ is available at time $t$ is $0$ when $t > C - c_j$. 
This is exactly a knapsack instance with deadline $C$.

\medskip
\noindent \textbf{Comparison to the Online Optimum.} As a consequence of the above discussion, we aim to design polynomial-time algorithms with provable guarantees (i.e., approximation algorithms) for the \ssx problem.
A natural approach would be to study \ssx under the competitive analysis
framework of \citet{borodin2005online}, where we compare ourselves to a
“clairvoyant” adversary who knows the random outcomes of the evaluation and expiration times in advance.
Unfortunately, the following example shows that an optimal algorithm  (even with unbounded computational power) obtains at most $O(1/{n})$ times the expected value obtained by an optimal clairvoyant (or, offline) algorithm.

\begin{example}
Consider $n$ identical options such that each option $j \in [n]$ has $v_j=1$, expires at time $n$ with probability $1$, and has evaluation time uniformly in $\{1, n\}$. Observe that, in expectation, $n/2$ options have a evaluation time of $1$, and a clairvoyant algorithm obtains expected total value of $n/2$. On the other hand, any algorithm that does not know the realizations of the evaluation times obtains expected total value of at most $2$ (in expectation, the second option that is selected has evaluation time equal to $n$).
\end{example}

This motivates us to utilize a more refined benchmark, and thus, we compare our algorithm against an optimal online (non-clairvoyant) algorithm.

\section{Missing Proofs From Section~\ref{sec:intro}}
Recall that we assume $\overline{F}_{S_j}(T) = 0$ for some $T \in \mathbb{N}$, and that $p_j(t) > 0$ for all $t \geq 1$ for each $j \in [n]$. 
In this section, we show that both of these assumptions can be made at the expense of an arbitrarily small loss in the objective value.

\subsection{Reduction to Bounded Support Evaluation Times}\label{app:bounded_support}

We begin by showing that we can, without loss of generality, assume that there exists $T$ such that $\overline{F}_{S_j}(T) = 0$ for each $j \in [n]$. 
Formally, we prove the following.

\begin{proposition}
    For any $\epsilon > 0$ and any instance of \ssx, there is a policy $\Pi$ such that $\Pi$ only selects options between times $1$ to $T$ and obtains value at least $(1 - \epsilon)$ times the value of the online optimum.
\end{proposition}

\begin{proof}
    Let $\OPT$ and $v(\OPT)$ denote the optimal policy and its value, respectively. 
    Furthermore, let $\tau \geq 1$ such that $\Pr(S_j > \tau) \leq \epsilon/n^2$ for all $j \in [n]$ and set $T = n\tau$. 
    Consider event $A = \{\exists j\in [n]: S_j>T \}$, and observe that $\Pr(A) \leq \epsilon/n$. 
    Let $B$ denote the event that  $\OPT$ selects an option at time after $T$. 
    Then, $\Pr(B) \leq \Pr(A) \leq \epsilon/n$. 
    Let $\Pi$ be the policy that follows $\OPT$ up to time $T$ (and after that it does not select any option). Let $v(\Pi)$ denote its expected total value. We have
    \begin{align*}
        v(\Pi) \,\, \geq \,\, v(\OPT) - v(\OPT\mid B)\cdot \Pr(B) \,\, \geq \,\, v(\OPT) - \epsilon \cdot v(\OPT) \,\, = \,\, (1 - \epsilon) \cdot v(\OPT),
    \end{align*}
    where, in the second inequality, we used that $v(\OPT \mid B) \leq n \cdot v_{\max}$ and $v_{\max} \leq v(\OPT)$.
\end{proof}



\subsection{Reduction to Expiration Times with Unbounded Support}\label{app:unbonuded-support}

In this section, we show that we can assume that $p_j(t) > 0$ for all $t \geq 1$ for all $j \in [n]$
at the expense of a small additive loss in the objective value.
Recall that $\Pr(E_j > t) = p_j(t)$, and let $G_j$ be a geometric random variable with probability of success $1-\epsilon$. 
That is, $\Pr(G_j = k) =\epsilon^{k-1}\cdot (1-\epsilon)$ for all $k \geq 1$. 
Consider modified expiration times $E_j'= \max\{E_j, G_j \}$. Then,
\[
\Pr({E}_j' < t) = \Pr(E_j < t) \cdot (1-\epsilon^{t-1}). 
\]
Thus, $p_j'(t):=  \Pr(E_j' > t) \geq \epsilon^{t-1}$ for all $t \geq 1$ and $j \in [n]$. 
Let $\OPT$ and $v(\OPT)$ denote the optimal policy and its value, respectively.
Let $\OPT'$ and $v'(\OPT')$ denote the optimal policy and its value with the modified expiration times ${E}_1',\ldots,{E}_n'$. The main result of this section is as follows.

\begin{proposition}
    For any $\epsilon > 0$, we have $v(\OPT) \leq v'({\OPT}') \leq (1 + n^2 \epsilon) \cdot v(\OPT)$.
\end{proposition}

\begin{proof}
    The first inequality is immediate as options remain longer in the system with the new expiration times ${E}_1',\ldots,{E}_n'$. 
    We prove the the second inequality as follows. 
    Let $B$ be the event that there exists some $j \in [n]$ such that $E_j \neq E_j'$. 
    Then, $\Pr(B) \leq \Pr(\exists j \in [n]: G_j \geq 2) \leq n \epsilon$. 
    Let $\Pi'$ be a policy with respect to expiration times $E_1',\ldots,E_n'$, and consider 
    the policy $\Pi$ that selects according to $\Pi'$ but
    with respect to the expiration times $E_1,\ldots,E_n$.
    Note that, conditioned on the complement event $\overline{B}$, the two policies $\Pi$ and $\Pi'$ obtain the same expected total value. 
    We denote by $v(\Pi)$ and $v'(\Pi')$ the expected total value obtained by $\Pi$ with respect to the expiration times $E_1,\ldots,E_n$ and  $E_1',\ldots,E_n'$, respectively. Then,
    \begin{align*}
        v(\Pi) &= \,\, \E_{E_1, \ldots, E_n}\left[ \sum_{j=1}^n v_j \cdot \mathbb{I}_{\{ \Pi \text{ selects }j \}}  \right] \,\, = \,\, \E_{E_1', \ldots, E_n'}\left[ \sum_{j=1}^n v_j \cdot \mathbb{I}_{\{ \Pi' \text{ selects }j \}} \mid \overline{B} \right] \cdot \Pr(\overline{B}) \\
        & = \,\, v'(\Pi') - \E_{E_1',\ldots, E_n'}\left[ \sum_{j=1}^n v_j \cdot \mathbb{I}_{\{ \Pi' \text{ selects }j \}} \mid {B} \right] \cdot \Pr({B}) \,\, \geq \,\, v'(\Pi') - n^2 \epsilon \cdot v(\OPT),
    \end{align*}
    where the final inequality used that $v_j\leq v(\OPT)$ for all $j\in [n]$. 
\end{proof}

\section{Estimating Probabilities in~\Cref{alg:random-serve}}\label{app:sampling}
Recall that our algorithm (\Cref{alg:random-serve}) requires computing the exact values of  $f_{j,t}$ for each option $j \in [n]$ and time $t \geq 1$.
We are not aware of a closed-form expression that can be used to compute these values; however, we will show that these values can be estimated well enough, which suffices for the algorithm. The main result of this section is as follows.

\begin{theorem}\label{thm:main-sampling}
    Suppose that \Cref{alg:random-serve} is an $\alpha$-approximation for the sequential selection with expiration problem when run with values exact $\{f_{j, t}\}_{j \in [n], t \geq 1}$ values.  
    Then, for any $\epsilon > 0$, there is an algorithm that, when run with approximate values $\left\{\overline{f}_{j, t}\right\}_{j \in [n], t \geq 1]}$ satisfying
    \(
    \left|f_{j, t} - \overline{f}_{j, t}\right| \,\, \leq \,\, \epsilon \cdot f_{j,t} \text{ for all } j\in [n], t \geq 1,
    \)
    obtains expected value at least $(1 - O(\epsilon))\cdot \alpha$ times the value of the optimal online policy.
\end{theorem}
\begin{proof}
    The modified algorithm is identical to \Cref{alg:random-serve} but with a slight modification in how options are added to the consideration set.
    At any time step $t$, if the DM is free, the algorithm adds option $j$ to the consideration set $\C^*$ with probability \({(1-\epsilon)^2\cdot x^*_{j, t}}/{2\cdot p_j(t)\cdot \overline{f}_{j, t}}\).

    The analysis is identical to the proof of \Cref{thm:appx-bound}, and begins by proving that the modified algorithm is feasible. 
    In particular, we need to show that ${(1-\epsilon)^2\cdot x^*_{j,t}}/{2\cdot p_j(t)\cdot \overline{f}_{j, t}} \leq 1$ for all $j \in [n]$ and $t \geq 1$. 
    Let $\run_{j,t}$ denote the event that the algorithm  selects option $j$ at time $t$ and let $\mathtt{C}_{j,t}$ denote the event that option $j$ is in $\C^*$ at time $t$.
    We begin by bounding  $\Pr(\mathtt{C}_{j, t})$. We have:
    \begin{align}
    \textstyle \Pr(\mathtt{C}_{j, t})  
    & \textstyle = \Pr(\mathtt{C}_{j, t} \mid \eA_{j, t}, \SF_t, \NC_{j, t}) \cdot \Pr(\SF_t, \NC_{j, t} \mid \eA_{j, t}) \cdot \Pr(\eA_{j, t})
    =  \frac{(1-\epsilon)^2\cdot x^*_{j, t}}{2 \cdot p_j(t) \cdot \overline{f}_{j, t}} \cdot f_{j, t} \cdot p_j(t) \leq \frac{(1-\epsilon)\cdot x^*_{j, t}}{2}\notag,
    \end{align}
    where the penultimate inequality follows from the algorithm design and the definition of $f_{j, t}$, and the final inequality $\overline{f}_{j, t} \geq (1-\epsilon)\cdot f_{j, t}$ for all $j \in [n]$ and $t \geq 1$. This allows us to conclude that for any $j \in [n]$ and $t \geq 1$, and $\tau < t$:
    \begin{enumerate}
        \item $\Pr(\run_{i,\tau}\mid \eA_{j,t}) \leq (1-\epsilon)\cdot x_{i,\tau}^*/2$ for all options $i \neq j$, and

        \item $\Pr(\mathtt{C}_{j,\tau}\mid \eA_{j,t})\leq \frac{(1-\epsilon)\cdot x_{j,\tau}^*}{2\cdot p_j(\tau)}$ for all options $j \in [n]$.
    \end{enumerate}
    We use these properties to complete the proof.
    To do so, we again consider the quantity $1 - f_{j, t}$, which can be upper bounded by the probability of events where, either (i) the DM is busy at time $t$ because of another option, or (ii) option $j$ was considered at an earlier time $\tau < t$. 
    By applying the union bound, we get
    \begin{align*}
    \textstyle 1-f_{j,t} \,\, & \textstyle \leq \,\, \sum_{\tau<t} \sum_{i\neq j} \Pr(\run_{i,\tau} \mid \eA_{j,t}) \cdot \Pr(S_i>t-\tau) + \sum_{\tau<t} \Pr(\mathtt{C}_{j,\tau} \mid \eA_{j,t})\\
    & \textstyle \leq \,\, \sum_{\tau<t} \sum_{i \neq j} \frac{(1-\epsilon)\cdot x^*_{i,\tau}}{2} \cdot \overline{F}_{S_i}(t-\tau) + \sum_{\tau < t} \frac{(1-\epsilon)\cdot x^*_{j, \tau}}{2\cdot p_j(\tau)} \,\, \leq \,\, \frac{1-\epsilon}{2}\left(1 + \left(1 -\frac{x^*_{j,t}}{p_j(t)}\right)\right),
    \end{align*}
    where the final inequality follows from the feasibility of $\x^*$. 
    Rearranging the inequality and using $(1-\epsilon)\cdot f_{j, t} \leq \overline{f}_{j, t}$ completes the argument.

    Finally, to obtain a bound on the expected value of the algorithm, it suffices to observe that 
    \[
    \textstyle \Pr(\mathtt{C}_{j,t}\mid \SF_t) \geq \frac{(1-\epsilon)^2 \cdot x^*_{j,t}}{2 \cdot (1+\epsilon) \cdot \Pr(\SF_t)}.
    \]
    Then, following the analysis in \Cref{lem:key} implies that the expected total value is at least $(1 - 3\epsilon)$ times the value obtained by \Cref{alg:random-serve}.
\end{proof}

We finish up this section by showing that we can estimate $\overline{f}_{j, t}$ values in polynomial-time.
To this end, we will use $N = O\left( \frac{v_j \cdot n^2 \cdot T\cdot \log(2n^3T)}{\epsilon^3}\right)$ independent samples to obtain an estimate of $f_{j, t}$ for option $j \in [n]$ at time step $t$. Note that the overall time taken for estimating the $f_{j, t}$ values is $O(nT \cdot N)$. 

Fix time step $t \geq 1$, and let $\overline{f}_{j, t}$ denote the average of the corresponding event obtained using $N$ samples. 
A crucial requirement for this simulation-based estimation scheme is a lower bound on the $f_{j, t}$ value. 
In particular, if $\widehat{f}_{j, t}$ is an independent sample, we have \(\E[\widehat{f}_{j, t}] = f_{j,t} \geq \frac{x^*_{j,t}}{2 \cdot p_j(t)}\). Thus, it suffices to assume that $x^*_{j, t} \geq \frac{\epsilon}{v_j \cdot n^2 \cdot T}$ for all $j \in [n]$ and $t \geq 1$. 
This assumption can be made without loss of generality, as it results in at most an $\epsilon$-additive loss in the total objective value.
On the plus side, we obtain the following uniform lower bound: $f_{j, t} \geq \frac{\epsilon}{2v_j \cdot n^2 \cdot T}$.
As $\overline{f}_{j, t}$ is an average of $N$ independent samples, each with mean $f_{j, t}$, using a Chernoff bound, we get, with probability at least $1 - \frac{1}{n^3T}$:
\[
\pr(|\overline{f}_{j, t} - f_{j, t}| \geq \epsilon \cdot f_{j, t}) \leq 2\exp\left(-\epsilon^2 \cdot \frac{N\cdot f_{j,t}}{3}\right) \leq 2\exp\left(-\epsilon^2 \cdot \frac{N\cdot \epsilon}{12}\right) \leq \frac{1}{n^3T}.
\]

We say that the sampling was successful if \( |\overline{f}_{j, t} - f_{j, t}| \leq \epsilon \cdot f_{j, t} \)  for all options $j \in [n]$ and for all time steps $t \geq 1$. 
By the union bound, sampling is successful with probability at least $1 - \frac{1}{n}$.
If sampling is not successful, then we obtain no value.
Hence, combined with \Cref{thm:main-sampling} and \Cref{thm:appx-bound}, the expected total value obtained is at least $(1-\frac{1}{n}) \cdot (1 - O(\epsilon)) \cdot \frac{1}{2}\cdot(1-\frac{1}{e}) \approx 0.316 - O(\epsilon)$.

\section{Greedy Policies for I.I.D. Evaluation Times}\label{app:greedy_analysis}

In this section, we formalize the high-level arguments discussed in \Cref{greedy_section} and give a formal proof of \Cref{thm:greedy}. Before providing the details, we need to introduce new definitions since the intermediate policy makes decisions based on the history of the computation. That is, the decision that $\ALG'$ makes at time $t$ depends on all the decisions and outcomes observed from times $1,\ldots,t-1$. 

\subsection{Preliminaries}

In this subsection, we recast the \ssx problem by first defining states that keep track of the history of actions and options remaining in the system. Then, we define policies that make decisions based on these states. Finally, we introduce the value (or value) collected by the policy.

The problem input are $n$ options with triple $(v_j,E_j,S_j)$ for $j\in [n]$. A \emph{state} $\bm{\sigma}$ for time $t$ is an ordered sequence of length $t-1$  of the form $\sigma_1 \sigma_2 \cdots \sigma_{t-1}$, where:
\begin{enumerate}
	\item Each $\sigma_\tau \in \{  (R_\tau, \mathtt{null}, \mathtt{null})  : R_\tau \subseteq [n] \} \cup \{  (R_\tau, s, \tau \to j)  : j\in [n], R_\tau \subseteq [n], s \geq 1 \}$, and\label{cond:valid_state_1}
	
	\item If $\sigma_\tau = ( R_\tau, s, \tau \to j )$,  then $\sigma_{\tau+1},\ldots, \sigma_{\min\{ \tau+s-1, t-1 \}}$ must be of the form $(R,\mathtt{null},\mathtt{null})$, and \label{cond:valid_state_2}
	
	\item $R_{t-1}\subseteq R_{t-2}\subseteq \cdots \subseteq R_1 \subseteq [n]$. We define $R_0=[n]$. \label{cond:valid_state_3}
\end{enumerate} 
The triplet $(R_\tau, s, \tau\to j)$ represents that at time $\tau$, the  decision-maker selects (or simulate) option $j$, the evaluation time of $j$ is $s$, and the options that remain have not expired the system are $R_\tau$. The triplet $(R_\tau,\mathtt{null}, \mathtt{null})$ indicates that at time $\tau$ a previous option is still evaluating or the decision-maker decided not to select any option. States have a natural recursive structure: $\bm{\sigma}=\bm{\sigma}'(R_{t-1}, s_{t-1}, t-1\to {j_{t-1}})$ or $\bm{\sigma}=\bm{\sigma}'(R_{t-1},\mathtt{null}, \mathtt{null})$, depending on the last triplet of $\bm{\sigma}$. The \emph{initial state} $\bm{\sigma}_0$ is the empty sequence. Given a state $\bm{\sigma}=\sigma_1\cdots\sigma_{t-1}$, we say that the state $\bm{\sigma}'=\sigma_1\cdots \sigma_\tau$ is a \emph{prefix} of $\bm{\sigma}$ for any $\tau\leq t-1$. Let  $\bm{\Sigma}_t\subseteq \bm{\sigma}$ be the set of all states for time $t$. Then, $\bm{\Sigma}_{1}=\{\bm{\sigma_0}\}$ and $\bm{\Sigma}=\bigcup_{t\geq 1} \bm{\Sigma}_t$ is the set of all possible states for any time $t\geq 1$.

A policy is a function $\Pi: \bm{\Sigma} \to [n]\cup\{\mathtt{null} \}$. Given a state $\bm{\sigma}=\sigma_1\cdots\sigma_{t-1}$ for time $t-1$, let $\tau_{\bm{\sigma}}=\max\{  \tau\leq t-1 : \sigma_{\tau}=(R_\tau, s, \tau \to j) \text{ for some }s,j \}$ the index of the last triplet $\sigma_{\tau}$ in $\bm{\sigma}$ that is of the form $(R_{\tau},s,\tau \to j)$; we define $\tau_{\bm{\sigma}}=-\infty$ if $\bm{\sigma}$ has no such triplets. Then, we define $s_{\bm{\sigma}}$ to be the evaluation time in the triplet $\sigma_{\tau_\sigma}$ if $\tau_\sigma\geq 1$; otherwise, $s_{\bm{\sigma}}=-\infty$. The policy $\Pi$ is \emph{feasible} if for any state $\bm{\sigma}$ for time $t$, we have that the sequences of $t$ triplets defined by $$\bm{\sigma}'=\begin{cases}
	\bm{\sigma}(R_{t} , s,  t\to j ), & \text{if } \Pi(\bm{\sigma}) =j \in [n], \text{ and } \tau_{\bm{\sigma}}+s_{\bm{\sigma}} \leq t \\
	\bm{\sigma}(R_t, \mathtt{null}, \mathtt{null}) & \text{if }\Pi(\bm{\sigma})=\mathtt{null} 
\end{cases}$$
is a state for time $t+1$ (i.e., it satisfies~\ref{cond:valid_state_1}-\ref{cond:valid_state_3}) for any subset $R_t\subseteq R_{t-1}$ and $s\geq 1$. In other words, a valid policy must respect the evaluation time of the last option select before $t$, if any, while also must respect the expirations, if any. A state $\bm{\sigma}$ for time $t$ is \emph{generated by a valid policy $\Pi$} if $\bm{\sigma}=\bm{\sigma}_0$ or $\bm{\sigma}$ can be obtained from its prefix $\bm{\sigma}'$ for time $t-1$ by applying the policy $\Pi$.

Given a state $\bm{\sigma} = \sigma_1 \ldots \sigma_{t-1}$ at time $t$, we say that \emph{option $j$ has been selected before time $t$ (in $\bm{\sigma}$)} if there exists a triplet in $\bm{\sigma}$ of the form $\sigma_\tau = (R_\tau, s_{i,\tau}, \tau \rightarrow j)$ with $j \in R_{\tau-1}$; the largest $\tau' \geq 1$ for which the prefix $\bm{\sigma}'$ up to time $\tau'$ in $\bm{\sigma}$, where option $j$ has not been selected before $\tau'$, represents the \emph{time when $j$ is selected in $\bm{\sigma}$}. Suppose that option $j$ is selected in $\bm{\sigma}$ at time $\tau'$; then, we say that any time $\tau > \tau'$ for which $\sigma_{\tau} = (R_{\tau}, s_{i,\tau}, \tau \rightarrow j)$ holds, represents a time when option $j$ is \emph{simulated}, regardless of whether $j$ belongs to $R_{\tau-1}$. For a valid policy $\Pi$ and a state $\bm{\sigma}$ for time $t$, we say that \emph{$\Pi$ selects option $j$ at time $t$} if $\Pi(\bm{\sigma})=j$, $j\in R_{t-1}$, and $j$ has not been selected before $t$ in $\bm{\sigma}$. If $\Pi(\bm{\sigma})=j$ and $j\notin R_{t-1}$ or $j$ has been selected before $t$ in $\bm{\sigma}$, then, we say that \emph{$\Pi$ simulates option $j$ at time $t$}.

Given a valid policy $\Pi$ and a state $\bm{\sigma}$ for time $t$, we define the expected value of the $\Pi$ starting from the state $\bm{\sigma}$ as 
\[
v( \bm{\sigma} \mid \Pi ) = \begin{cases}
	 v_j \mathbb{I}_{\{  j\in R_{t-1}, j \text{ has not been selected in }\bm{\sigma} \}} + \E_{S_j, R_t}\left[   v(\bm{\sigma} (R_t, S_j, t\to j))   \right] & \text{if } \Pi(\bm{\sigma}) = j\in [n] \\
  
	 \E_{R_t} \left[  v(\bm{\sigma}(R_t, \mathtt{null}, \mathtt{null} ))  \right] &  \text{if }\Pi(\bm{\sigma}) =\mathtt{null}
\end{cases}
\]
Note that a policy can only collect a value from an option $j$ once and for the first $k$ options that selects. After that, if the policy decides to select $j$ again, it can only simulate selecting the option. Given the initial state $\bm{\sigma}_0$, we are interested in solving
$$\sup_{\substack{\Pi \text{ feasible}  \\\text{policy}} } v(\bm{\sigma}_0\mid \Pi).$$
This value is well-defined, as the policy that maps everything to $\mathtt{null}$ is feasible (with an expected value of $0$). Additionally, the value of any policy is upper-bounded by $\sum_{j\in [n]} v_j$. We denote by $\OPT$ the feasible policy that maximizes $v(\bm{\sigma} \mid \Pi)$. Note that an optimal policy will never simulate an option if it can select an option. However, allowing this possibility is going to be useful for the next subsection when we analyze the greedy algorithm. When it is clear from the context, and to agree with the notation introduced in previous subsections, we will write $v(\Pi)$ to refer to $v(\bm{\sigma} \mid \Pi)$.
 
\begin{remark}
    Notice that we have only defined \emph{deterministic policies}. Alternatively, we can define randomized policies as $\Pi(\bm{\sigma}) \in \Delta([n]\cup{\mathtt{null}})$, where $\Delta(X) = \{ \mathbf{x} \in \mathbb{R}^X: \sum_{j\in X} x_j = 1, \mathbf{x}\geq 0 \}$ represents the set of distributions over a finite set $X$. A standard result from Markov decision process theory ensures that any randomized policy has a deterministic counterpart with the same value. For more details, refer to~\cite{puterman2014markov}.
\end{remark}

\subsection{Analysis of Greedy}


\paragraph{The intermediate policy $\ALG'$} We construct $\ALG'$ iteratively from $\OPT$ as follows. Let $\ALG^{0}= \OPT$ be the initial policy. Let  $\bm{\Sigma}_t\subseteq \bm{\sigma}$ be the set of all states for time $t$. Then, $\bm{\Sigma}_{1}=\{\bm{\sigma_0}\}$ and $\bm{\Sigma}=\bigcup_{t\geq 1} \bm{\Sigma}_t$. Since each $\bm{\Sigma}_t$ is countable, we can assume there is an enumeration $\bm{\Sigma}_{t}=\{  \bm{\sigma}^1, \bm{\sigma}^2, \ldots \}$. 

For each $t=1,2,\ldots$, let $\ALG^t = \ALG^{t-1}$. Now, we iteratively modify $\ALG^t$ for each $\bm{\sigma}\in \bm{\Sigma}_t$. For each $t'=1,2,\ldots$, take $\bm{\sigma}^{t'}\in \bm{\Sigma}_t$ and let $R_{t-1}\subseteq [n]$ be the set associated with the last triplet appearing in $\bm{\sigma}^{t'}$ (for $t=1$ we set $R_0=[n]$).
	
	If $R_{t-1}=\emptyset$ or $\ALG^{t}(\bm{\sigma}^{t'})=\mathtt{null}$, then we go to $t'+1$. Otherwise, let $i^*$ be the index of the highest-valued option $v_{i^*}=\max_{i\in R_{t-1}} v_i$ and $\ALG^t(\bm{\sigma}^{t'})=j$. 
	\begin{itemize}
		\item If $j=i^*$, then, go to $t+1$. In this case $\ALG^t$ is already behaving like greedy.
		
		\item If $j\neq i^*$, then, set $\ALG^t(\bm{\sigma}^{t'})=i^*$. In this case, $\ALG^t$ will follow the greedy action. Now, for every $\bm{\sigma}'\in \bigcup_{\tau> t}\bm{\Sigma}_{\tau}$ that has $\bm{\sigma}$ as a prefix, say $\bm{\sigma}' = \bm{\sigma} \sigma_{t} \sigma_{t+1} \cdots \sigma_{\tau-1}$, we redefine $\ALG^t(\bm{\sigma}'')= \ALG^{t-1}(\bm{\sigma}') $, where $\bm{\sigma}'' = \bm{\sigma}(R_t,s, t\to i^*) \sigma_{t+1}\cdots \sigma_{\tau-1}$, where $\sigma_t=(R_t,s,t\to j)$. In other words, after $t$, $\ALG^t$ will follow $\ALG^{t-1}$ as if $\ALG^{t-1}$ had chosen $i^*$ at time $t$ in state $\bm{\sigma}$. 
	\end{itemize}

Note that in this construction, we potentially need to examine an infinite number of states in $\Sigma_{t}$. We can overcome this challenge as follows: Let $T\geq 1$ be the smallest value for which $\Pr(S \geq T+1) \leq \varepsilon/n$, where $\varepsilon>0$. Then, define $\OPT'$ as the policy that follows $\OPT$ up to states for time $T$, and for states at time $\tau > T$, $\OPT'$ only maps the state to $\mathtt{null}$. It can be easily demonstrated that $\OPT'$ guarantees an expected value of at least $(1-\varepsilon)v(\OPT)$. Now, by employing the same procedure with $\OPT'$, we can focus on states for time $t\leq T$ that include triplets $(R,s,\tau\to j)$ with $s\leq T$. Since we can make $\varepsilon$ arbitrarily small, the above construction holds.

\begin{proposition}\label{prop:from_opt_to_intermediate}
	We have $v(\ALG')\geq v(\OPT)/2$.
\end{proposition}

\begin{proof}
	For the sake of simplicity, we run the argument over $\OPT'$ that does not make any decision after time $T$ so the construction above only generates a finite sequence of algorithms $\ALG^0, \ldots, \ALG^T$, and $\ALG'=\ALG^T$. We know that up to an arbitrarily small error, the result will hold. We denote $\OPT'$ simply by $\OPT$. 
	
	In the proof, we compare the value collected by $\ALG^t$ against $\ALG^{t-1}$ for each $t\geq 1$. To do this, we take a state $\bm{\sigma}$ for time $T+1$ that is generated by $\OPT$. From this state, we generate a state $\bm{\sigma}^T$ that is generated by $\ALG^T$. This process shows that the probability that $\OPT$ generates $\bm{\sigma}$ is equal to the probability that $\ALG^T$ generates $\bm{\sigma}^T$. Furthermore, if $v_{\bm{\sigma}}$ denotes the value collected by options selected in $\bm{\sigma}$ and $v_{\bm{\sigma}^T}$ the value collected by selections in $\bm{\sigma}^T$, then, we show that $v_{\bm{\sigma}}\leq 2 v_{\bm{\sigma}^T}$. With these two ingredients, we conclude that twice the expected value collected by $\ALG^T=\ALG'$ upper bounds the expected value of $\OPT$. The rest of the proof is a formalization of this idea.
	
	Let $\bm{\sigma}=\sigma_1 \sigma_2\cdots \sigma_T$ be a state for time $T+1$ generated by $\OPT$.  Let $i_1, \ldots, i_k$ be the options selected by $\OPT$ in this state and $\tau_1< \ldots<\tau_k$ be the times they are selected. We first generate a sequence of states for time $T+1$ that $\ALG^t$ can generate for each $t=1,\ldots,T$. Let $\bm{\sigma}^0=\bm{\sigma}$. For $t\geq 1$, we generate $\bm{\sigma}^t$ from $\bm{\sigma}^{t-1}$ as follows. Let $\sigma_\tau^{t-1}$ be the $\tau$-th triplet in $\bm{\sigma}^{t-1}$. Then,
	\begin{itemize}
		\item If $\sigma_t^{t-1}=(R_t,\mathtt{null},\mathtt{null})$, then, $\bm{\sigma}^t=\bm{\sigma}^{t-1}$ and go to $t+1$.
		
		\item Otherwise, if $\sigma_t^{t-1}=(R_t,s,t\to j)$, then, let $\ALG^{t}(\sigma_1^{t-1}\cdots \sigma_{t-1}^t)=j'$ and define $\bm{\sigma}^t = \sigma_1^{t-1}\cdots \sigma_{t-1}^{t-1} (R_t,s,t\to j') \sigma_{t+1}^{t-1}\cdots \sigma_{T}^{t-1}$. That is, $\bm{\sigma}^t$ is the same as $\bm{\sigma}^{t-1}$ with the $t$-th triplet changed to $(R_t,s,t\to j')$.
	\end{itemize}
	With this, we have generated a sequence $\bm{\sigma}^t$ that $\ALG^t$ can generate. An inductive argument shows that the probability that the sequence $\bm{\sigma}$ is generated by $\OPT$ is the same as the probability that $\ALG^t$ generated $\bm{\sigma}^t$. Moreover, notice that $\ALG^t$ and $\ALG^{t-1}$ make decisions different from $\mathtt{null}$ at the same times $\tau_1,\ldots, \tau_k$.
	
	Now, we are ready to compare the values of $\ALG^{t-1}$ and $\ALG^t$. We do this by recording the values as in the table in Figure~\ref{fig:iterative_construction}. The columns corresponds to the times where $\OPT,\ALG^1,\ldots,\ALG^T$ make decisions, in this case, $\tau_1,\ldots,\tau_k$. The rows are the values that each algorithm $\ALG^t$ collects with a minor twist. For each time $\tau\leq t_t$, we boost the value collected by $\ALG^t$ by a factor of $2$ as Figure~\ref{fig:iterative_construction} shows.
	\begin{figure}[ht]
		\centering
		\begin{tabular}{|c|c|c|c|c|c|c|}
			\hline
		time & $\tau_1$ & $\tau_2$ & $\tau_3$ & $\tau_4$ & $\cdots$  & $\tau_k$ \\
		\hline\hline
		$\OPT$ &  $v_{i_1}$ & $v_{i_2}$ & $v_{i_3}$ & $v_{i_4}$ & $\cdots$ & $v_{i_k}$  \\ \hline
		$\ALG^1$ & $2v_{i^*_1}$  & $v_{i_2}$ & $0$ & $v_{i_4}$ & $\cdots$ & $v_{i_k}$\\ \hline
		$\ALG^2$ & $2v_{i^*_{1}}$ & $v_{i_2}$ & $0$ & $v_{i_4}$ & $\cdots$ & $v_{i_k}$\\ \hline
		$\ALG^3$ & $2v_{i^*_{1}}$ & $2v_{i^*_2}$ & $2v_{i_3^*}$ & $v_{i_4}$ & $\cdots$ & $v_{i_k}$\\ \hline
		$\vdots$  & $\vdots$ & $\vdots$ & $\vdots$ & $\vdots$ & & $\vdots$ \\ \hline
		$\ALG^T$ & $2v_{i^*_{1}}$ & $2v_{i^*_2}$ & $2 v_{i_3^*}$ & $2v_{i^*_4}$ & $\cdots$ & $0$\\ \hline
	\end{tabular}
	\caption{Comparison of coupled values. In this example, $1=t_1< 2 < t_2$ and so $\ALG^2$ does not change the action in $t_2$.}\label{fig:iterative_construction}
	\end{figure}
	We claim that the sum of values in the $t$-th row in Figure~\ref{fig:iterative_construction} upper bounds the sum of values in the $(t-1)$-th row for all $t\geq 1$. Indeed, for $t\geq 1$, if $t\neq \tau_1,\ldots,\tau_k$, then the sum of values in row $t$ and $t-1$ are the same. if $t=\tau_j$, for some $i$, then, $v_{i_j^*}$ upper bounds the value obtained in the same column for row $t-1$. Furthermore, the additional copy induced by the factor of $2$ in $v_{i_j^*}$ helps pay for any future occurrence of item $i_j^*$ in row $t-1$. 
    
	In summary, we obtain
	$$\sum_{j \text{ selected in  }\bm{\sigma}}v_j \Pr(\OPT \text{ generates }\bm{\sigma}) \leq 2 \sum_{j \text{ selected in  }\bm{\sigma}^T}v_j \Pr(\ALG^T\text{ generates }\bm{\sigma}^T)$$
	and summing over all $\bm{\sigma}$ (which has a unique $\bm{\sigma}^T$ associated) concludes that $v(\OPT)\leq 2 v(\ALG')$.
\end{proof}

\begin{proposition}\label{prop:from_inter_to_greedy}
	We have $v(\ALG_{\text{Greedy}}) \geq v(\ALG')$.
\end{proposition}

\begin{proof}
    Observe that whenever $\ALG'$ selects an option, it does so with the highest-valued one available. Greedy does the same, but it does not miss any opportunity to select an option as $\ALG'$ might do. Hence, the expected value collected by greedy can only be larger than the expected value collected by $\ALG'$. This finishes the proof.
\end{proof}

\section{Missing Details of Computational Results}\label{app:comp-sec}

We begin by giving a detailed description of the \batch algorithm used in our computational setup.

\medskip
\noindent \textbf{ENS Algorithm.}
Inspired by active search applications, we adapt and test a variant of the efficient nonmyopic active search (\batch) method proposed by \citet{jiang2017efficient} 
to align with our budgeted sequential search setting. 
The \batch algorithm is a greedy selection strategy that, whenever the DM is free,  
selects the option that maximizes the expected reward attainable under the remaining budget $b$ and the current set of available jobs $R$. 
Let the initial budget be $B$. We define auxiliary variables
\[r_k^{\text{safe}}(b') := v_k \cdot \Pr(S_k \le b') \cdot \Pr(E_k \geq B-b'\mid E_k \geq B-b), \ w_k^{\text{safe}}(b') =  \Pr(S_k \leq b') \cdot \mathbb{E}[S_k \mid S_k \le b'].\]
Here $r_k^{\text{safe}}(b')$ roughly measures the expected value we would obtain from selecting option $k$ and $w_k^{\text{safe}}(b')$ is roughly evaluating the expected processing time of option $k$ under the ``good'' event that its service time is less than remaining budget $b'$.
The estimated value of selecting option $j$ with budget $b$ and available set $R$ is then given by:
\[V_j(b, R)  = \sum_{\tau} \Pr(S_j = \tau) \cdot \mathbb{I}\{\tau \le b\} \Big[ v_j + K \big(b-\tau, R\setminus\{j\}\big) \Big],\] 
where the exploration term $K$ is computed approximately by solving the following knapsack problem:
\[K(b', R') \, \approx \, \max_{x_k \in \{0,1\}}  \;\sum_{k \in R'} x_k \cdot r_k^{\text{safe}}(b') \quad \text{s.t.} \quad \sum_{k \in R'} x_k \cdot w_k^{\text{safe}}(b') \le b'.\]
In this formulation, $v_j$ represents the immediate reward from selecting option $j$, while 
$K$ heuristically estimates the residual value obtainable from the remaining budget and options.
The final policy chooses the option $j \in [n]$ that maximizes $V_j(B, [n])$.


\subsection{Missing Plots from \Cref{sec:call-center}}\label{app:comp-sec-cc}
\begin{figure}[!h]
  \centering
  \begin{subfigure}[b]{0.45\textwidth}
  \centering
  \includegraphics[width=\textwidth]{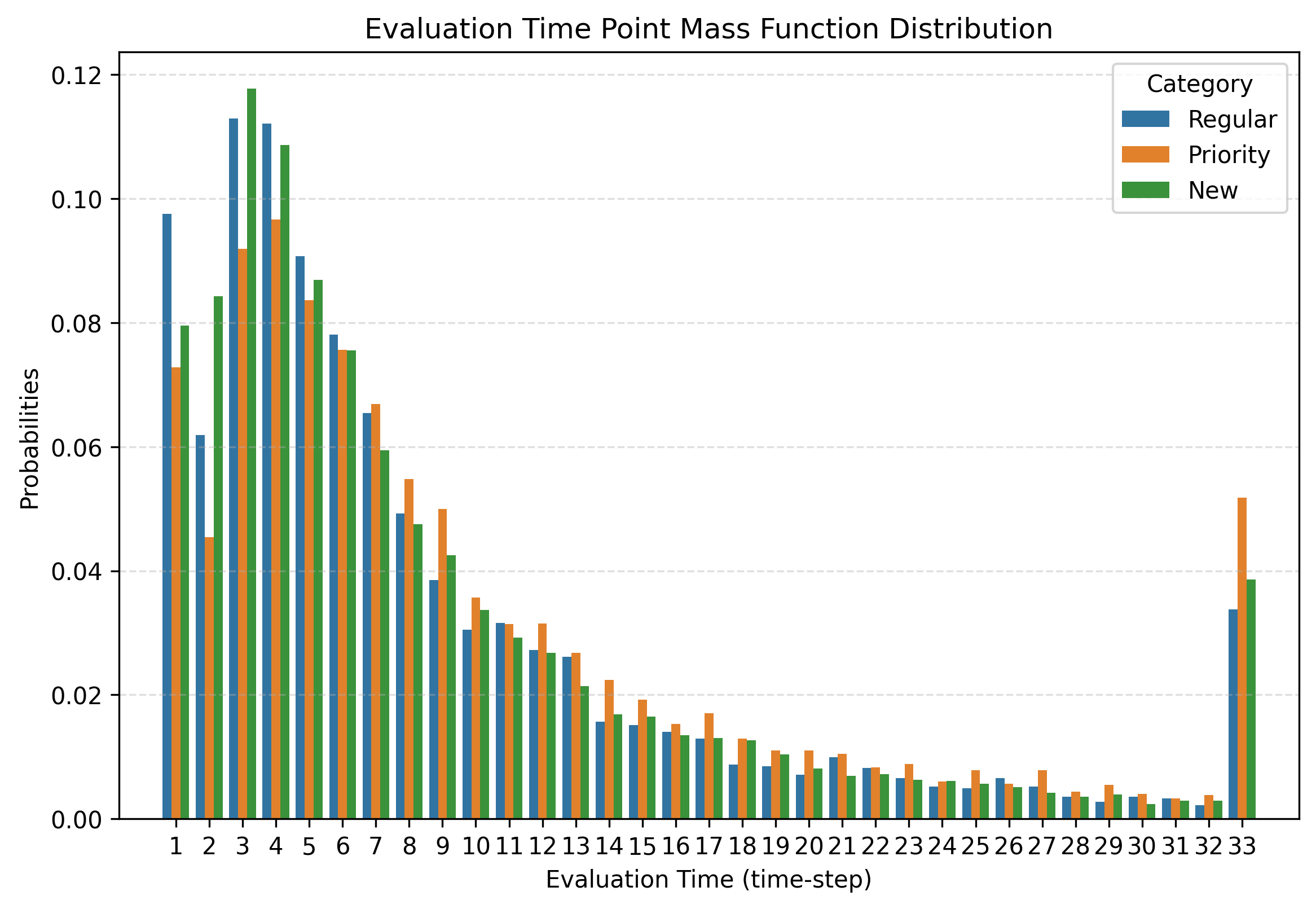}
  \caption{Evaluation time point mass function distribution for the three categories of customers.}
  \label{pic:service}
  \end{subfigure}
  \quad
  \begin{subfigure}[b]{0.45\textwidth}
  \centering
  \includegraphics[width=\textwidth]{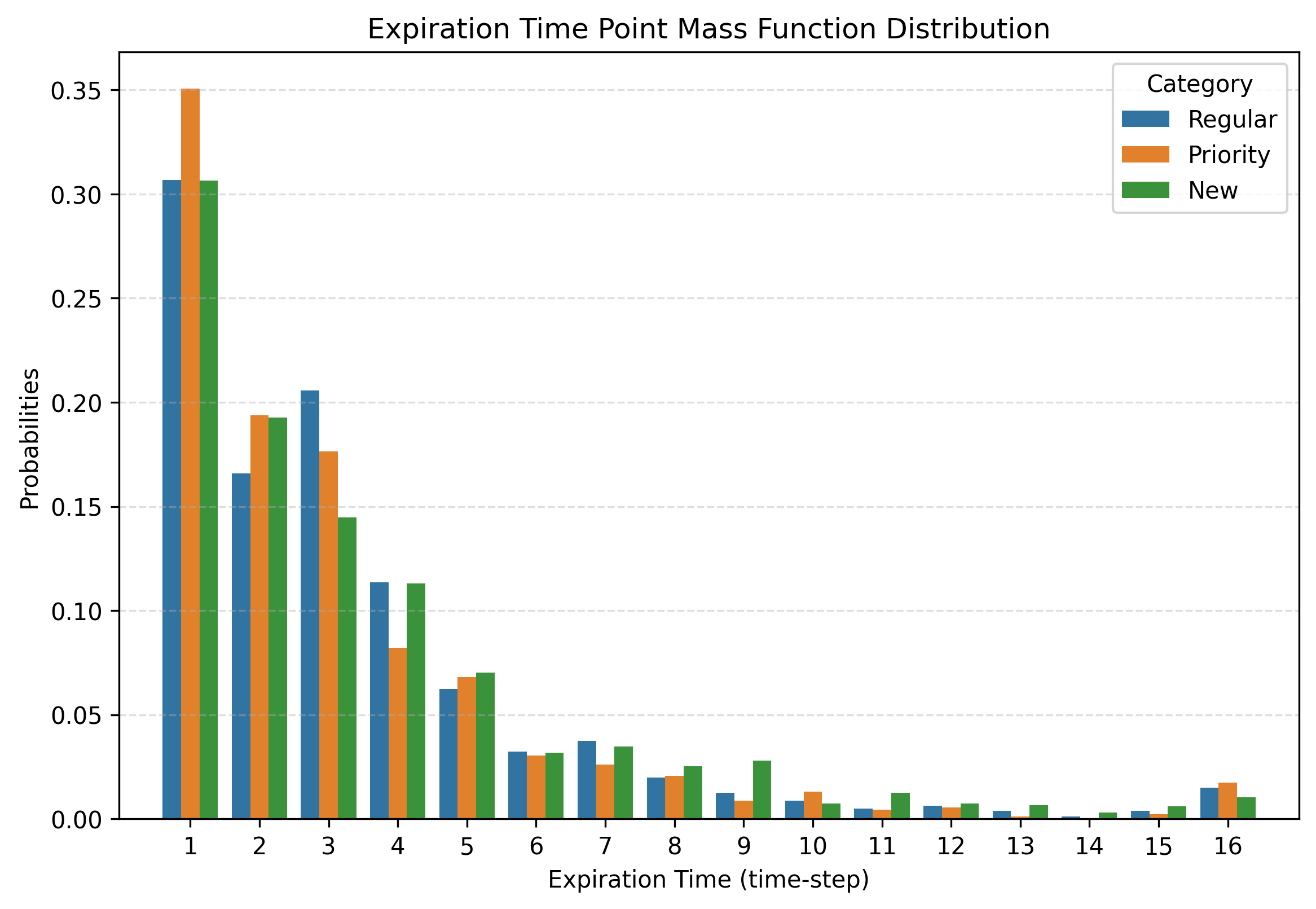}
  \caption{Expiration time point mass function distribution for the three categories of customers.}
  \label{pic:departure}
  \end{subfigure}
  \caption{}
  \label{}
\end{figure}

\ignore{
\begin{figure}[p]
  \centering
  \includegraphics[width=3.5in]{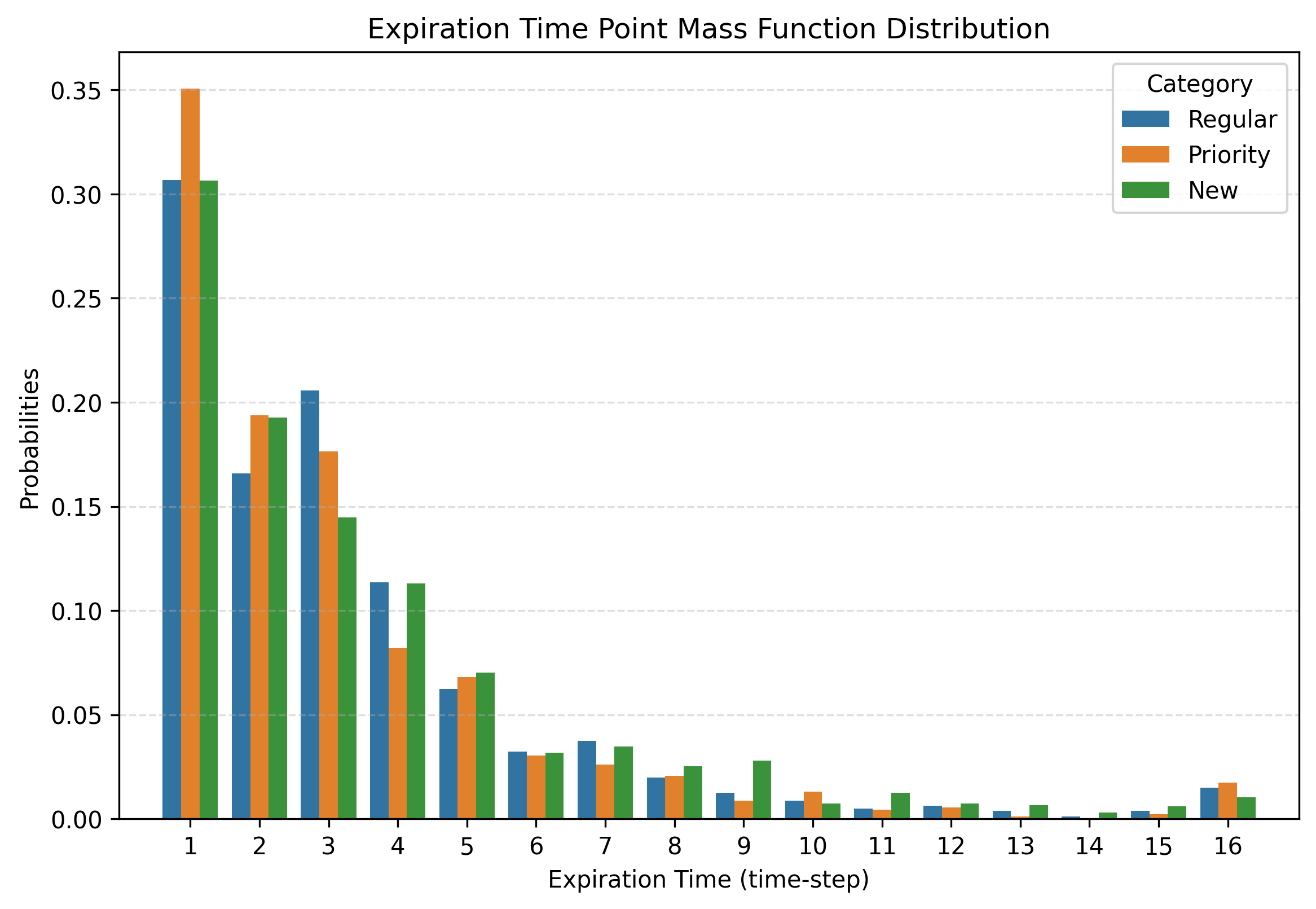}
  \caption{Expiration time point mass function distribution for the three categories of customers.}
  \label{pic:departure}
\end{figure}
}

\ignore{
\begin{figure}[p]
  \centering
  \small
  \includegraphics[width=0.8\linewidth]{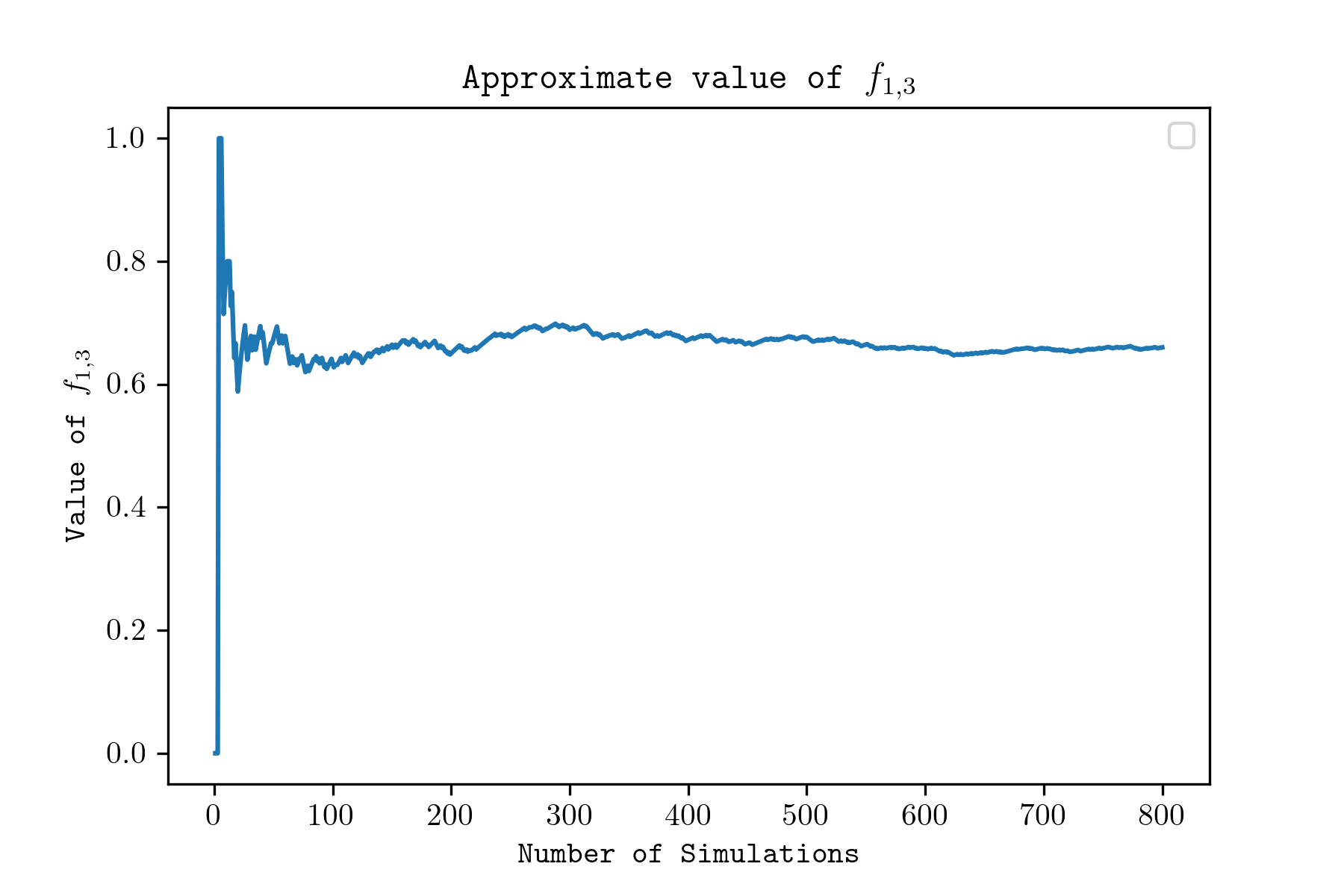}
  \caption{Sample Path of Approximating $f_{1,3}$}
  \label{fig:sample_path_approx}
\end{figure}
}

\end{document}